\documentclass[11pt]{amsart}
\usepackage{amsmath}
\usepackage{amssymb}
\usepackage{amsthm}
\usepackage{array}
\usepackage{xy}
\usepackage[pdftex]{graphicx}
\usepackage{hyperref}
\usepackage{color}
\usepackage{transparent}
\usepackage{latexsym} 
 
\theoremstyle{plain}
  \newtheorem{theorem}{Theorem}[section]
  \newtheorem{lemma}{Lemma}[section]
  \newtheorem{proposition}{Proposition}[section]
  
  \newtheorem{corollary}{Corollary}[section]
  \newtheorem{definition}{Definition}[section]
  \newtheorem{remark}{Remark}[section]

   \newcommand{\beqn}{\begin{eqnarray}}
   \newcommand{\eeqn}{\end{eqnarray}}
   \newcommand{\beqs}{\begin{eqnarray*}}
   \newcommand{\eeqs}{\end{eqnarray*}}
   \newcommand{\ban}{\begin{eqnarray*}}
   \newcommand{\nan}{\end{eqnarray*}}
   \newcommand{\beq}{\begin{equation}}
   \newcommand{\eeq}{\end{equation}}

\renewcommand{\det}{\mbox{det}}

\newcommand{\R}{\mathbb{R}}

\numberwithin{equation}{section}

  \numberwithin{equation}{section}
  \numberwithin{figure}{section}

\begin{document}

\title[Regularity of singular set in optimal transportation]{\textbf{regularity of singular set in optimal transportation}}
 
\author[S. Chen]
{Shibing Chen}
\address [Shibing Chen]
{School of Mathematical Sciences,
University of Science and Technology of China,
Hefei, Anhui 230026, China}
\email{chenshib@ustc.edu.cn}

\author[J. Liu]
{Jiakun Liu}
\address [Jiakun Liu]
	{School of Mathematics and Statistics,
	The University of Sydney,
	Camperdown, NSW 2006, Australia}
\email{jiakun.liu@sydney.edu.au}

\thanks{Research of Chen was supported by National Key R and D program of China 2022YFA1005400,
 2020YFA0713100, National Science Fund for Distinguished Young Scholars
(No.12225111), NSFC No.12426202, NSFC No.12141105.
Research of Liu was supported by Australian Research Council DP230100499 and FT220100368. }



\date{}

\begin{abstract}  
In this paper, we establish a regularity theory for the optimal transport problem when the target is composed of two disjoint convex domains. This is an important model in which singularities arise.
Even though the singular set does not exhibit any form of convexity a priori, we  prove its higher order regularity by developing novel methods, which also have many other applications.
Notably, our results are achieved without requiring any convexity of the source domain. 
This aligns with Caffarelli's celebrated regularity theory.
\end{abstract}

\maketitle

\baselineskip=16.4pt
\parskip=3pt

\section{Introduction}
Let $f, g \in L^1(\mathbb{R}^n)$ represent two probability densities concentrated on bounded open sets $\Omega, \Omega^*\subset\mathbb{R}^n$, respectively. Suppose there exists a positive constant \(\lambda\) such that $\frac{1}{\lambda} < f, g < \lambda$ in $\Omega, \Omega^*$, respectively.
Invoking Brenier's theorem \cite{Br1}, we identify two globally Lipschitz convex functions, $u$ and $v$, both defined on $\mathbb{R}^n$, satisfying the following conditions:
\begin{align}
    (Du)_\sharp f = g, & \text{ with } Du(x) \in \overline{\Omega^*} \text{ for almost every } x \in \mathbb{R}^n; \label{condi1}\\
    (Dv)_\sharp g = f, & \text{ with } Dv(y) \in \overline{\Omega} \text{ for almost every } y \in \mathbb{R}^n. \label{condi2}
\end{align}
From the regularity theory of Caffarelli \cite{C92}, if $\Omega^*$ is convex, then the potential function $u$ belongs to $C^{1,\beta}_{\text{loc}}(\Omega)$ for some $\beta\in(0,1)$ and is strictly convex inside $\Omega$.

If $\Omega^*$ consists of two disjoint convex domains, as seen in the example of \cite{C92} when mapping the unit disc onto two shifted half-discs, there appears a singular set $\mathcal{F}\subset\Omega$ on which $u$ is not differentiable. 
In this paper, considering the case when $\Omega^* = \Omega^*_1 \cup \Omega^*_2$, with $\Omega^*_1$ and $\Omega^*_2$ being bounded convex domains separated by a hyperplane $\mathcal{H}$, we will establish the regularity of the singular set $\mathcal{F}$.

From previous works \cite{CL1, KM1}, one knows $\mathcal{F} \subset \Omega$ is a Lipschitz hypersurface, partitioning $\Omega$ into two subdomains $\Omega_i$, $i = 1, 2$, such that 
	\begin{equation}\label{condi3} 
		(Du)_{\sharp}(f\chi_{\Omega_i}) = g\chi_{\Omega^*_i}\qquad \text{ for }\quad i = 1, 2.
	\end{equation}
Moreover, $\mathcal{F}$ is characterised as $\partial\Omega_i \cap \Omega$ for $i = 1, 2$, which can be expressed as the graph of a certain Lipschitz function.

Define the functions $u_i : \mathbb{R}^n\rightarrow \mathbb{R}$, for $i = 1, 2$, as follows:
	\begin{equation}\label{uidef}
    		u_i(x) := \sup\{L(x) : L \text{ is affine, } L \leq u \text{ in } \Omega_i, \text{ and } DL(x) \in \Omega_i^*\} \quad \forall x\in \mathbb{R}^n.
	\end{equation}
From \cite{CL1, KM1}, one has $\mathcal{F} = \{u_1 = u_2\} \cap \Omega$. 
Hence, if $u_i$ are differentiable at $x \in \mathcal{F}$, the unit normal vector to the hypersurface $\mathcal{F}$ at $x$ satisfies
\begin{equation}\label{gn1} 
    \nu_\mathcal{F}(x) = \frac{Du_1(x) - Du_2(x)}{|Du_1(x) - Du_2(x)|}.
\end{equation}
Under the additional assumption of strict convexity for $\Omega$, $\Omega^*_1$, and $\Omega^*_2$, the hypersurface $\mathcal{F}$ has been proven to be locally $C^{1,\beta}$ regular for some $\beta \in (0, 1)$, see \cite{CL1, KM1}.
We remark that the corresponding $C^{1,\beta}$ regularity of the free boundary in the optimal partial transport was obtained in \cite{CM,I}.

\vskip5pt
Notably, the domain convexity plays an important role in previous works.
Indeed, if $\Omega$ is convex, one can invoke Caffarelli's regularity theory \cite{C92} to derive the following key property:
\begin{equation}\label{alex1}
\frac{1}{\lambda^2}\chi_{{\Omega^*}}\leq \det D^2v\leq \lambda^2\chi{_{\Omega^*}}
\end{equation}
interpreted in the Alexandrov sense. Thanks to \eqref{alex1}, by using a localisation lemma one can apply the methods from \cite{C92b} to obtain a quantitative strict convexity estimate for $v$ near $Du_i(x_0)$ for $x_0\in\mathcal{F}$, which in turn implies $C^{1,\beta}$ regularity of $u_i$, $i=1,2$, \cite{CL1, KM1}. 
However, the non-convexity of $\Omega$ renders \eqref{alex1} inapplicable, creating a significant obstacle in deriving the $C^{1,\beta}$ regularity as outlined above.

In this paper, we employ some new ideas and techniques to overcome these difficulties.
Our first result is the $C^{1,\beta}$ regularity of $\mathcal{F}$ without necessitating any convexity of $\Omega$.

\begin{theorem}\label{t111}
Let $\Omega\subset\mathbb{R}^n$ be a bounded domain, and let $\Omega^* = \Omega^*_1 \cup \Omega^*_2$ be a union of bounded convex domains $\Omega^*_1$ and $\Omega^*_2$ in $\mathbb{R}^n$, separated by a hyperplane $\mathcal{H}$. If $\frac{1}{\lambda} < f, g < \lambda$ within $\Omega$ and $\Omega^*$, respectively, for some positive constant $\lambda$, then the hypersurface $\mathcal{F}$ (as long as it is non-empty) is $C^{1,\beta}$ regular inside $\Omega$ for some $\beta \in (0, 1)$. Additionally, $u_i \in C_{\text{loc}}^{1,\beta}(\Omega \cap \overline{\Omega_i})$ for $i = 1, 2$.
\end{theorem}

\vskip5pt
A much more challenging problem concerns the higher order regularity of $\mathcal{F}$, which is precisely solved in the main theorem of this paper:
\begin{theorem}\label{main1}
Under the assumptions of Theorem \ref{t111}, and further assuming that $\Omega^*_i$ are $C^2$ and uniformly convex for $i=1,2$, and the densities $f\in C_{\text{loc}}^\alpha(\Omega)$ and $g \in C^\alpha(\overline{\Omega^*})$ for some $\alpha\in(0,1)$, it follows that the singular set $\mathcal{F}$  (as long as it is non-empty) is $C^{2,\alpha}$ smooth. Moreover, $u_i \in C_{\text{loc}}^{2,\alpha}(\Omega \cap \overline{\Omega_i})$ for $i = 1, 2$.
\end{theorem}

It is worth noting that when the densities $f$ and $g$ are merely continuous, our method implies that the singular set $\mathcal{F}$ is $C^{1,1-\epsilon}$ for any $\epsilon>0$.

Note that for each $i=1,2$, $u_i$ satisfies the second boundary value problem:
\begin{equation*}
\left\{ 
\begin{array}{rl}
  \det\, D^2 u_i \!\!&= \ \frac{f}{g(D u_i)} \quad \text{in } \Omega_i, \\
  D u_i(\Omega_i) \!\!&= \ \Omega^*_i.
\end{array}
\right.
\end{equation*} 
From \eqref{gn1}, it suffices to prove $u_i$ is $C^{2,\alpha}$ up to the singular set $\mathcal{F} \subset \partial \Omega_i$. 
All previous related works require some sort of convexity of domains: In the study of global regularity of the optimal transport map, Caffarelli \cite{C96} requires the domains to be uniformly convex (see also \cite{D91, U1}). Recently, in \cite{CLW1} Chen-Liu-Wang relax this condition from uniform convexity to merely convexity. In the study of optimal partial transport \cite{CLW3}, the semi-convexity of the free boundary plays a crucial role, which follows from the interior ball property \cite{CM,AF09,AFi}. 
The most striking point of this paper is that we manage to establish the higher order regularity of \( u_i \) up to the singular set \( \mathcal{F} \), which does not exhibit any form of convexity a priori, see Figure 1.1.

A crucial ingredient in the proof of Theorem \ref{main1} is the following obliqueness estimate:
\begin{proposition}\label{uniob}
Under the hypotheses of Theorem~\ref{main1}, assume that $0 \in \mathcal{F}$. Denote by $y_0 = Du_1(0) \in \partial\Omega^*_1$ and $\hat{y}_0 = Du_2(0) \in \partial\Omega^*_2$. Let $\nu, \nu^*, \hat{\nu}^*$ be the unit inner normals to $\Omega_1, \Omega_1^*, \Omega_2^*$ at $0, y_0, \hat{y}_0$, respectively. Then $\nu \cdot \nu^* > 0$ and $-\nu \cdot \hat{\nu}^* > 0$.
\end{proposition}

\begin{figure}[h] \label{f1.1}
 \centering
 \includegraphics[scale=0.12]{1singular.png}
 \caption{The singular set $\mathcal{F}$ in non-convex domain $\Omega$. The function $v$ is $C^1$ near $y_0$, but has singularity in $\Omega^*_1$. }
\end{figure}

In previous works such as \cite{C96}, \cite{CLW1} and \cite{SY}, the obliqueness estimate was proved relying on the convexity of the domains. 
The exploration continued in \cite{CLW3}, where the authors investigated the optimal partial transport problem and successfully established the obliqueness estimate at points on the free boundary. 
Notably, even though the free boundary is not convex in general, the interior ball property implies the semi-convexity of the free boundary, which enables the authors to make substantial progress in \cite{CLW3}.

However, the interior ball property is not applicable to the problem considered in this paper. 
In fact, due to the absence of any variant of convexity of the singular set, all existing methods \cite{C96,CLW1,CLW3,SY} do not apply. 
Therefore, new observations and ideas are needed to establish the obliqueness estimate.  
In order to prove Proposition \ref{uniob}, it suffices to rule out the following three distinct scenarios:
\begin{itemize}
\item[]\textbf{Case I:} $\nu\cdot \nu^*=0,\ \hat\nu \cdot \hat\nu^* = 0$, and $\nu, \nu^*, \hat{\nu}^*$ are non-coplanar;
\item[]\textbf{Case II:} $\nu\cdot \nu^*=0,\ \hat\nu \cdot \hat\nu^* = 0$, and $\nu, \nu^*, \hat\nu^*$ are coplanar;
\item[]\textbf{Case III:} $\nu\cdot \nu^*=0,\ \hat\nu \cdot \hat\nu^* > 0,$
\end{itemize}
where $\hat\nu:=-\nu$ denotes the unit inner normal of $\Omega_2$ at $0.$
Although a detailed proof will be given in Section \ref{S4}, here we would like to summarise some innovative ideas for each case. 

\textit{For Case \textbf{I} and \textbf{III}}: The approach is to show the splitting behaviour of the singular set and $\partial\Omega_1^*$ during a blow-up procedure. The strategy unfolds as follows:
\begin{enumerate}
     \item By a proper affine transformation, we can make the three normals $\nu, \nu^*, \hat\nu^*$ mutually perpendicular to each other. By estimating the shape of sub-level sets of \(v\), we have a \((3+\epsilon)\)-uniform convexity estimate of \(v\) at \(y_0 = Du_1(0)\) and \(\hat y_0 = Du_2(0)\). Utilising the duality between \(u\) and \(v\), we then obtain \(C^{1,\frac{1}{2}-\epsilon}\) estimates of \(u_1\) and \(u_2\) at $0$. In the light of \eqref{gn1}, we derive a \(C^{1,\frac{1}{2}-\epsilon}\) estimate of \(\mathcal{F}\) at $0$. Additionally, by combining the shape estimate on the sub-level set of $v$ at $y_0$ with the tangential \(C^{1, 1-\epsilon}\) estimate of \(v\) at \(\hat y_0\) in the \(\nu^*\)-direction, we obtain a rather unexpected ``\emph{above the tangent}'' property, namely \(\mathcal{F}\cap\text{span}\{\nu, \nu^*\}\) is above its tangent plane at \(0.\)

    \item 
    By applying the first blow-up, we can obtain that $\mathcal{F}$ becomes flat in $n-2$ directions in the limit profile. Here, the term blow-up refers to the limit of a sequence of sets normalised by some affine transformations. The shape estimates for the sub-level set of $v$ control these affine transformations, which surprisingly provide just enough information to confirm the flatness of $\mathcal{F}$ in $n-2$ directions in the limit. This success is largely attributable to the $C^{1,\frac{1}{2}-\epsilon}$ estimate of $\mathcal{F}$ at $0$. The ``\emph{above the tangent}'' property then guarantees that the limit domains are well-positioned. 

    \item 
    By applying the second blow-up, the limit of $\Omega_1$ becomes a cylindrical shape, and the limit of $\partial\Omega_1^*$ becomes the graph of a non-negative quadratic polynomial. Subsequently, after the secondary blow-up, we can employ techniques from \cite{CLW1, CLW3} to demonstrate the transformation of both $\Omega_1$ and $\Omega_1^*$ into cylindrical configurations, analogous to the scenarios depicted in \cite{CLW1}.
\end{enumerate}

\textit{For Case \textbf{II}}: Observe that $\nu^*$ is parallel to $\hat\nu^*$. We can prove the width of the centred section of $v$ in the $\nu^*$-direction is very small. On one hand, by combining this width estimate with \eqref{gn1} and the disjointness of $\Omega_1^*$ and $\Omega_2^*$, we can show that $\mathcal{F}$ becomes progressively flatter in the $\nu^*$-direction during the blow-up process. On the other hand, we establish a strict convexity estimate for the normalisation of $u_1$ that is independent of the blow-up scale, which however contradicts the flattening of $\mathcal{F}$.

The new ideas and methods developed in this paper are also useful for investigating many other interesting problems. In \cite{CL2}, we applied these methods to establish the obliqueness estimate at the intersection points of the free and fixed boundaries, which leads to the sharp global Sobolev regularity of the optimal map between active regions. More recently, in \cite{CLL}, we utilised these methods to investigate the partial regularity of free boundaries in optimal partial transportation between non-convex polygonal domains. In particular, we demonstrate that the free boundary is smooth except at a finite number of singular points. Additionally, in \cite{CLL}, we use these methods to provide a precise description of the singular set of the optimal map between non-convex polygonal domains, by showing that the singular set is a smooth curve away from a finite number of points.

The structure of this paper is organised as follows: Section \ref{S2} introduces important notations and definitions used for the subsequent discussions. Section \ref{S3} is dedicated to establishing the $C^{1,\beta}$ regularity of the singular set $\mathcal{F}$. Section \ref{S4} focuses on the proof of the obliqueness estimate, a crucial step in our analysis. Finally, Section \ref{S5} outlines the proof of the $C^{2,\alpha}$ regularity of the singular set $\mathcal{F}$.

\vskip5pt 

\noindent\emph{Acknowledgements:} We sincerely appreciate Professor Xu-Jia Wang for introducing this problem to us and for many inspiring discussions.

\section{Preliminaries}\label{S2} 

Given a convex function $w: \mathbb{R}^n \rightarrow \mathbb{R}$, its subdifferential at $x \in \mathbb{R}^n$ is defined as:
\begin{equation}
\partial^{-}w(x) := \{y \in \mathbb{R}^n : w(z) \geq w(x) + y \cdot (z - x) \ \text{ for all} \ z \in \mathbb{R}^n\}.
\end{equation}
For any subset $A\subset\mathbb{R}^n$, $\partial^{-}w(A):=\cup_{x\in A}\partial^{-}w(x)$.

Let $U \subset \mathbb{R}^n$ be an open set and $C$ be a positive constant. 
We say that a convex function $w$ satisfies the Monge-Ampère inequality:
\begin{equation}
\frac{1}{C} \chi_{_U} \leq \det\, D^2w < C \chi_{_U}
\end{equation}
in the Alexandrov sense if and only if for any Borel set $A$,
\begin{equation}\label{aama}
\frac{1}{C} |A \cap U| \leq |\partial^-w(A)| \leq C |A \cap U|,
\end{equation}
where $|\cdot|$ denotes the $n$-dimensional volume. 

Now,
let us consider the functions and sets $u, v, u_i$ and $\Omega_i, \Omega^*_i$ for $i=1, 2$ as delineated in the introduction. 
Specifically, we have the conditions \eqref{condi1}--\eqref{condi3} and $u_i$ satisfies \eqref{uidef}. 
It is important to note that $\Omega$ can be expressed as the union $\Omega = \Omega_1\cup \Omega_2\cup\mathcal{F}$, and within each $\Omega_i$, $i=1,2$, we have $u_i = u$. 
Additionally, we have the following dual relationships:
\begin{align}
	 v(y) &= \sup_{x\in \Omega}\{y\cdot x - u(x)\} \quad \text{ for all } y\in \mathbb{R}^n; \label{dual111} \\
	 u_i(x) &= \sup_{y\in \Omega_i^*}\{y\cdot x - v(y)\} \quad \text{ for all } x\in \mathbb{R}^n,\ \ i=1,2. \label{dual222}
\end{align} 

Given that $\Omega^*_i$ are convex for $i=1, 2$, Caffarelli's regularity theory \cite{C92} ensures that $u_i\in C^{1,\beta}_{\text{loc}}(\Omega_i)$ for some $\beta\in(0,1)$, and they satisfy the Monge-Ampère inequality in the sense of Alexandrov:
\begin{equation}\label{Asolv} 
    C^{-1}\chi_{\Omega_i} \leq \det\, D^2 u_i \leq C\chi_{\Omega_i} \quad \text{ for } i=1, 2.
\end{equation}
Lastly, we note that $\mathcal{F}=\{x\in\Omega: u_1(x)=u_2(x)\}$.

To proceed further, we introduce some useful definitions and notations.
Let $w:\mathbb{R}^n\rightarrow \mathbb{R}$ be a convex function whose graph does not contain any infinite straight line.

\begin{definition}\label{defS}
Given a point $y_0\in \mathbb{R}^n$ and a small positive constant $h$, we define the centred sections of the convex function $w$ at $y_0$ with height $h$ as:
\begin{equation}\label{sect}
S^c_{h}[w](y_0) := \left\{y\in\mathbb{R}^n : w(y) < w(y_0) + (y-y_0)\cdot \bar{p} + h\right\},
\end{equation}
where $\bar{p}\in \mathbb{R}^n$ is chosen such that the centre of mass of $S^c_{h}[w](y_0)$ coincides with $y_0$.
Additionally, we define the sub-level set of $w$ at $y_0$ with height $h$ as:
\begin{equation}\label{sub}
S_h[w](y_0) := \left\{y\in \mathbb{R}^n : w(y) < \ell_{y_0}(y) + h\right\},
\end{equation}
where $\ell_{y_0}$ denotes an affine supporting function of $w$ at $y_0$.
\end{definition}

\begin{remark}
\emph{For the existence of $\bar{p}$ in \eqref{sect}, we refer the reader to \cite{C92b}.
If $w(0)=0$ and $w\geq 0$, by \cite[Remark 2.5]{CLW3}, we have that 
\begin{equation}\label{secrela}
w \leq Ch \quad \text{in}\ S^c_h[w](0),
\end{equation}
where $C$ is a constant depending only on $n$. }
\end{remark}

In subsequent discussions, a constant will be referred to as universal if it depends solely on the parameters $n, \Omega, \Omega^*, \lambda$, and $\text{dist}(x_0, \partial\Omega)$, where $x_0$ is a generic point of $\mathcal{F}$. Additionally, we will employ the notation $C_n$ to represent a positive constant dependent exclusively on the dimension $n$, with the acknowledgment that its value may change across different contexts. The vectors \( e_i \), where \( i = 1, \ldots, n \), will be used to denote the standard basis vectors in the coordinate system.

\section{$C^{1,\beta}$ regularity of $\mathcal{F}.$}\label{S3}

In this section, we continue to utilise the notations established in Section \ref{S2} for $\Omega, \Omega^*, \Omega^*_i$ $(i=1,2)$, $u, v, u_i, \mathcal{F}$. Consider a specific point $x_0\in\mathcal{F} \subset \Omega$. Select sequences $\{x_k\} \subset \Omega_1$ and $\{\hat{x}_k\} \subset \Omega_2$, both converging to $x_0$ as $k \to \infty$. By applying a translation of coordinates, we may assume $x_0 = 0$.
Without loss of generality, and potentially passing to subsequences, we assume $\lim_{k\to\infty}Du(x_k)=y_0 \in \overline{\Omega^*_1}$, and $\lim_{k\to\infty}Du(\hat{x}_k)=\hat{y}_0 \in \overline{\Omega^*_2}$, respectively. This ensures that $y_0, \hat{y}_0 \in \partial^{-}u(0).$ By \eqref{dual111}, we deduce that $0\in\partial^{-}v(y_0) \cap \partial^{-}v(\hat{y}_0)$.

By subtracting an appropriate constant, we may assume $v(y_0)=v(\hat{y}_0)=0$ and $v\geq0$. Define $\Sigma = \{y\in\mathbb{R}^n : v(y) = 0\}$ and denote the set of extreme points of \( \Sigma \) as \( \text{ext}(\Sigma) \).  

Recall Equation \eqref{condi2}, especially $Dv(y) \in \overline{\Omega}$ for almost every $y\in \mathbb{R}^n.$ 
Consider a section of \(v\), namely \(Z:=\{y \in \mathbb{R}^n : v(y) < \ell(y)\}\), with \(\ell\) being an affine function. Provided that \(Z\) is bounded and contains a point \(y\) such that \(0 \in \partial^{-}v(y)\), since \(0\in \mathcal{F}\) is inside the interior of \(\Omega,\)
 we can assert that \(\partial^{-}v(y) \cap \Omega \neq\emptyset\). 
 Denoting the centre of mass of \(Z\) by \(z\), and employing John's lemma, we identify an affine transformation \(L\) such that $L(z)=0$ and \(B_1(0) \subset LZ \subset B_{C_n}(0)\), where \(C_n\) solely depends on the dimension \(n\). 

From \cite[Inequality (4)]{FK1}, we deduce that:
\begin{equation}\label{la1}
|\inf_{Z} (v-\ell)|^n \geq C_1 \left| \left( \frac{1}{2}Z \right)_{z} \cap \Omega^* \right| |Z|,
\end{equation}
where $\left(\frac{1}{2}Z\right)_{z}$ represents the dilation of $Z$ with respect to $z$, and $C_1$ is a constant depending on $n$ and $\lambda$.
From \cite[Proposition 1]{FK1}, we have the following bound:
\begin{equation}\label{ua1}
|v(y) - \ell(y)|^n \leq C_2 |Z \cap \Omega^*||Z| \, \text{dist}\big(L(y), \partial(L(Z))\big) \quad\mbox{for } y\in Z,
\end{equation}
where $C_2$ is a constant depending only on $n$, $\lambda$, and the diameter of $\Omega$.

\begin{lemma}\label{sc111}
The set $\Sigma$ is both bounded and closed, and $\text{\normalfont ext}(\Sigma)\subset \partial\Omega^*$.
\end{lemma}

\begin{proof}
First, we prove that $\Sigma$ is bounded and closed. Assume, to the contrary, that this is not the case. 
By convexity, $\Sigma$ would then contain a half-line of the form $\{q + te : t \geq 0\}$ for some $q \in \Sigma$ and $e \in \mathbb{S}^{n-1}$. For any $y \in \mathbb{R}^n$ and $x \in \partial^{-}v(y)$, by the convexity of $v$ one has 
	$$ (x - 0) \cdot (y - q - te) \geq 0. $$ 
As $t$ tends to infinity, this implies that $x \cdot e \leq 0$. 
Given the arbitrariness of $y$, it follows that 
	$$\partial^{-} v(\mathbb{R}^n) \subseteq \{x \cdot e \leq 0\}.$$ 
Since we have $(Dv)_{\sharp}g = f$, it can be deduced that $\Omega \subseteq \partial^{-} v(\mathbb{R}^n) \subseteq \{x \cdot e \leq 0\},$ which contradicts the fact that $0 \in \Omega$. Thus, $\Sigma$ must be bounded and closed.

Next, we prove that $\text{ext}(\Sigma) \subseteq \partial\Omega^*$. Suppose to the contrary that there exists a point $p \in \text{ext}(\Sigma)$ such that $p$ is either in $\Omega^*$ or in $\mathbb{R}^n \setminus \overline{\Omega^*}$. If $p \in \text{ext}(\Sigma) \cap \Omega^*$, then we have $0 \in \partial^{-}v(p) \cap \Omega$. Referring to Proposition 2 and the subsequent discussion in \cite{FK1}, $v$ is inferred to be $C^1$ and strictly convex in a neighborhood of $p$. Consequently, by duality, the function $u$ must be differentiable at $0$, leading to a contradiction. In the case where $p \in \text{\normalfont ext}(\Sigma) \cap (\mathbb{R}^n \setminus \overline{\Omega^*})$, the application of inequalities \eqref{la1} and \eqref{ua1}, along with the arguments presented in Section 3.1 of \cite{FK1}, also yields a contradiction. Hence, we conclude that $\text{ext}(\Sigma) \subset \partial\Omega^*$.
\end{proof}

Since exposed points are dense in the set of extreme points, we can choose points 
	\begin{equation}\label{exppt}
		p\in\text{ext}(\Sigma) \cap \partial\Omega^*_1 \quad \text{ and }\ \  \hat{p} \in  \text{ext}(\Sigma) \cap \partial\Omega^*_2
	\end{equation} 
 such that both are exposed points of \( \Sigma \). In this context, we establish the following localisation properties.

\begin{lemma}\label{locc1}
Given any \( r_0 > 0 \), there exists an \( h_0 > 0 \) such that for all \( h \leq h_0 \), the inclusion \( S^c_h[v](p) \subset B_{r_0}(p) \) holds,  where $p$ is as in \eqref{exppt}. 
\end{lemma}

\begin{proof}
Since \( p \) is an exposed point of \( \Sigma \), there exists a unit vector $e\in\mathbb{S}^{n-1}$ such that 
	\begin{align*} 
		& \Sigma \subset \{ y : (y - p) \cdot e \leq 0 \}, \\
		& \Sigma \cap \{ y : (y - p) \cdot e = 0 \} = \{ p \}.
	\end{align*} 
Since $S^c_h[v](p)$ is centred at $p$, if $S^c_h[v](p) \cap \{ y : (y - p) \cdot e \geq 0 \} \subset B_r(p) $, then it follows that $S^c_h[v](p) \subset B_{C_n r}(p)$.

Observing that \( v > 0 \) in \( \mathbb{R}^n \setminus \Sigma \) and given any small constant \( r_0>0 \), the convexity and continuity of \( v \) ensure the existence of a constant \( \delta_0 \) such that
\begin{equation}\label{pfloc1}
v > \delta_0 \quad \text{in} \ \ \{ y : (y - p) \cdot e \geq 0 \} \setminus B_{\frac{r_0}{C_n}}(p).
\end{equation}

By \eqref{secrela} we have $v \leq C_n h$ in $S^c_h[v](p)$. Setting \( h_0 = \frac{\delta_0}{C_n} \), for all \( h \leq h_0 \), we then have \( v \leq \delta_0 \) in \( S^c_h[v](p) \). From \eqref{pfloc1}, it follows that
\[ S^c_h[v](p) \cap \{ y : (y - p) \cdot e \geq 0 \} \subset B_{\frac{r_0}{C_n}}(p) \quad \forall\ h \leq h_0. \]
Therefore, we conclude that \( S^c_h[v](p) \subset B_{r_0}(p) \) whenever \( h \leq h_0 \).
\end{proof}

The next lemma establishes the localisation of centred sections near the point $p$. 

\begin{lemma}\label{locc2}
For any given \( \frac{1}{2} > r_0 > 0 \), let \( h_0 \) be the constant specified in Lemma \ref{locc1}. There exists a positive constant \( \delta_0 > 0 \) such that for all points \( y \in \overline{\Omega_1^*} \cap B_{\delta_0}(p) \) and for all \( h \leq h_0 \), the inclusion \( S_h^c[v](y) \subset B_{2r_0}(p) \) holds,  where $p$ is as in \eqref{exppt}. 
\end{lemma}

\begin{proof}
Assume, for the sake of contradiction, that the statement is false. Then, there exist sequences \( \{y_k\} \subset \overline{\Omega^{*}_1} \) converging to \( p \) and \( \{h_k\} \) with \( h_k \leq h_0 \), such that 
\begin{equation}\label{cseg1}
S_{h_k}^c[v](y_k) \cap \left(\mathbb{R}^n \setminus B_{2r_0}(p)\right) \neq \emptyset.
\end{equation}
Given that \( S_{h_k}^c[v](y_k) \) is balanced with respect to \( y_k \), there exists a line segment \( I_k \subset S_{h_k}^c[v](y_k) \) centred at \( y_k \) satisfying \( |I_k| \geq C_n r_0 \).
Without loss of generality, by taking a subsequence if necessary, we can assume that \( h_k \to \bar{h} \leq h_0 \) and \( I_k \to I_\infty \) as \( k \to \infty \), where \( I_\infty \) is a segment centred at \( p \) and satisfying
\begin{equation}\label{cseg2}
|I_\infty| \geq C_n r_0.
\end{equation}

First, we claim that $S_{h_k}^c[v](y_k)$ is uniformly bounded. 
From Lemma \ref{sc111}, $\Sigma$ is compact. 
Owing to the non-negativity and convexity of $v$, we deduce the existence of a large constant \( R > 1 \) satisfying
\[ v(y) > h_0 + \sup_{B_1(p)} v \quad \forall\, y \in \mathbb{R}^n \setminus B_R(p). \]

Suppose \( S_{h_k}^c[v](y_k) = \{y \in \mathbb{R}^n : v(y) < \ell_k\} \), for some affine function \( \ell_k \)  that fulfils \( \ell_k(y_k) = v(y_k) + h_k \). For any given unit vector \( e \in \mathbb{S}^n \), we may assume $\ell_k$ is decreasing in the $e$-direction, otherwise consider $-e$ instead.

Since \( y_k \) converges to \( p\), we may assume that \(y_k\in B_1(p) \) for all $k.$
It can be verified that \( y_k + 3Re \in \mathbb{R}^n \setminus B_R(p) \), leading to the inequality
\[ v(y_k + 3Re) > h_0 + \sup_{B_1(p)} v \geq \ell_k(y_k)\geq \ell_k(y_k+3Re). \]
This implies \( y_k + 3Re \notin S_{h_k}^c[v](y_k) \). 
Since \( S_{h_k}^c[v](y_k) \) is centred at \( y_k \), we deduce \( y_k - 3C_nRe \notin S_{h_k}^c[v](y_k) \). Given that \( e \) is arbitrary, we conclude the claim that 
\begin{equation} 
	S_{h_k}^c[v](y_k) \subset B_{3C_nR}(y_k) \subset B_{3C_nR+1}(p).
\end{equation}

Now, let's consider a subsequence (without changing notation) such that \( S_{h_k}^c[v](y_k) \) converges to a limit convex set \( S_\infty \) in the Hausdorff distance, and $h_k\to\bar h\leq h_0$ as $k\to\infty$. 
Write \( \ell_k = x_k \cdot (y - y_k) + h_k \) with \( x_k \in \mathbb{R}^n \), the global Lipschitz continuity of \( v \) ensures that \( \sup_k |x_k| \leq \|v\|_{\text{Lip}} < \infty \). Consequently, after possibly taking another subsequence, we can assert that \( x_k \to \bar{x} \) as \( k \to \infty \).

In the case when \( \bar{h} > 0 \), it follows that
\[
S_\infty = \left\{ y \in \mathbb{R}^n : v(y) < \bar{x} \cdot (y - p) + \bar{h} \right\} = S^c_{\bar{h}}[v](p).
\]
By \eqref{cseg1}, passing to limit we have
\[
S^c_{\bar{h}}[v](p) \cap \left( \mathbb{R}^n \setminus B_{2r_0}(p) \right) \neq \emptyset,
\]
leading to a contradiction with Lemma \ref{locc1}.

In the case when \( \bar{h} = 0 \), we observe that \( 0 \leq v \leq \bar{x} \cdot (y - p) \) for all \( y \) in \( I_\infty \). Given that \( I_\infty \) is centred at \( p \), it follows that \( v \) must be identically zero on \( I_\infty \), implying that \( I_\infty \subseteq \Sigma \). This contradicts to the fact that $p$ is an exposed point of $\Sigma$.
\end{proof}

Take $r_0$ small so that $B_{2r_0}(p)\cap \Omega_2^*=\emptyset$. Let $h_0, \delta_0$ be as in Lemma \ref{locc1} and \ref{locc2}. Define $\rho_0:=\frac{1}{2}\text{dist}(0,\partial \Omega)$. 
We will show the differentiability of $v$ near $p$ in the following lemma.
\begin{lemma}\label{c1lemma}
There exists a constant $\bar{r}_1<\frac{1}{3}\delta_0$, such that $v$ is differentiable at all $y\in \overline{\Omega_1^*}\cap B_{\bar{r}_1}(p)$, where $p$ is as in \eqref{exppt}. 
\end{lemma}

\begin{proof}

First, we establish the differentiability of \( v \) at the point \( p \).
Before giving the detailed proof, let us outline the main idea. Suppose \( v \) is not differentiable at \( p \). By subtracting an affine function, we may assume that $v$ satisfies \eqref{difs1} for some unit vector $e$ pointing outside $\Omega^*_1$. Consider the section \( S^c_h[v](y_{t_0}) = \{ v < \ell \} \), centred at \( y_{t_0} := p - t_0 e \in\Omega^*_1 \), where $t_0>0$ is chosen small so that $y_{t_0}$ lies sufficiently close to $p$. We have two observations: $(i)$  $p$ is nearly the minimum point of $v-\ell,$ and $(ii)$ After applying a normalisation via the affine transformation \( L \) in \eqref{normdd}, \( L(p) \) lies very close to the boundary of \( L\left(S^c_h[v](y_{t_0})\right) \) as $h\to0$. Using the estimates in \eqref{la1} and \eqref{ua1}, we arrive at a contradiction. This argument resembles Caffarelli's localisation technique and the discussion around equations (1.4) and (1.5) in \cite{GK}, but involves new delicate estimates  as $p\in\partial\Omega^*_1$ lies on the boundary. Below, we provide the detailed proof.

Consider \( n \) linearly independent unit vectors \( \hat{e}_i \) (may not be orthogonal), where \( i = 1, \ldots, n \). 
By convexity of $v$, it suffices to show that \( v(p + t\hat{e}_i) \) is differentiable at \( t = 0 \) as a convex function of a single variable \( t \), for each \( i = 1, \ldots, n \).

Suppose to the contrary that \( v \) is not differentiable at \( p \). Then there exists a unit vector \( e \), such that \( \{p + te : t < 0\} \cap \Omega_1^* \neq \emptyset\), and \( v(p + te) \) is not differentiable at \( t = 0 \). By subtracting an affine function \( \ell_1 \), we can assume the following behaviour (see Figure 3.1):
\begin{equation}\label{difs1}
\begin{aligned}
v &\geq 0&& \text{on}\ \mathbb{R}^n,\\
v(p + te) &= o(t) && \text{for } t < 0, \\
v(p + te) &= at + o(t) && \text{for } t > 0,
\end{aligned}
\end{equation}
where \( a>0 \) is a constant. 
Note that the target of \( Dv \) is changed to \( \tilde{\Omega} = \Omega - \{D\ell_1\} \), a translation of \( \Omega \).
Given that \( \partial^{-}v(p) \cap \tilde{\Omega} \) is non-empty, the estimates \eqref{la1}--\eqref{ua1} still apply.  

Define \( y_{t_0} := p - t_0e \) for $t_0 \in (0, \frac{1}{2}\delta_0)$ sufficiently small. 
Let's consider the centred section 
	$$S^c_{h}[v](y_{t_0}) = \{v < \ell\},$$ 
where $\ell$ is an affine function satisfying $\ell(y_{t_0}) = v(y_{t_0})+h$. 
Consider the intersection
of the boundary $\partial S^c_{h}[v](y_{t_0})$ with the line $\{p+te : t\in\mathbb{R}\}$, and let
	\[ \partial S^c_{h}[v](y_{t_0}) \cap \big\{p+te : t\in\mathbb{R}\big\} =  \big\{ p - t_1(h)e, \ \ p+t_2(h)e \big\}, \]
where $t_1(h) > t_0$ and $t_2(h) > -t_0$. 
By \cite[Lemma A.8]{CM}, we have that \(S^c_{h}[v](y_{t_0})\) varies continuously with respect to \(h \in (0, h_0]\). Hence, \(t_1(h)\) and \(t_2(h)\) depend continuously on \(h\).
Since \(S^c_{h}[v](y_{t_0})\) is centred at $y_{t_0}$, we have the following bounds:
\begin{equation}\label{t0t1}
t_0 < t_1(h) < C_n(t_0 + t_2(h)),
\end{equation}
where \(C_n\) is a constant depending only on \(n\).

\begin{figure}\label{f3.1}
 \centering
 \includegraphics[scale=0.2]{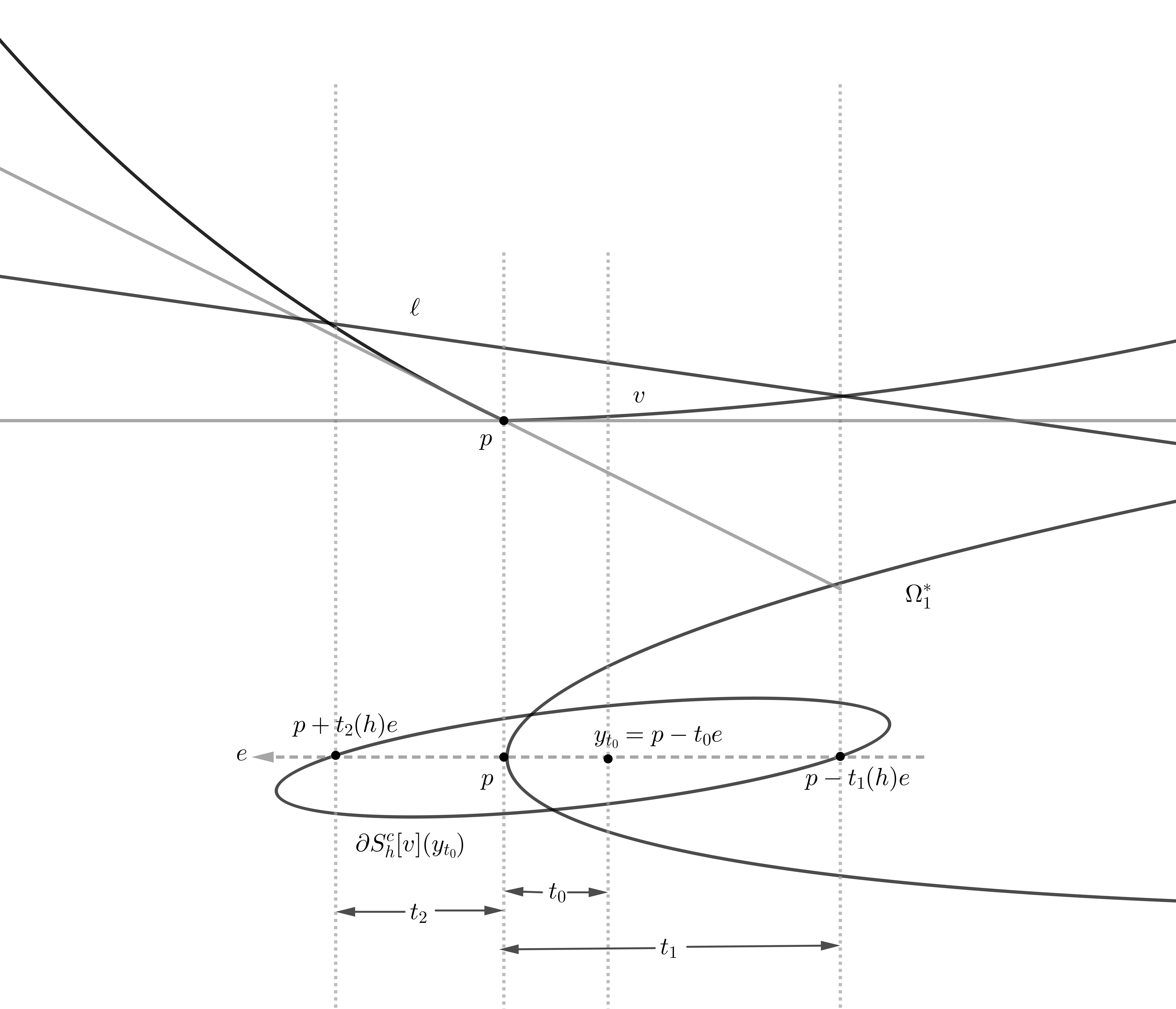}
 \caption{A centred section \(S^c_{h}[v](y_{t_0})\), near $p$. }
\end{figure}

We first \emph{claim} that \(t_2(h_0) \geq t_0\), provided $t_0$ is chosen sufficiently small initially.
Suppose, to the contrary, that \(t_2(h_0) < t_0\).
Then, by \eqref{difs1}, we have 
\begin{equation}\label{t21}
\ell(p+t_2(h_0)e) = v(p+t_2(h_0)e) \leq 2at_0,
\end{equation}
provided \(t_0\) is small enough.
By \eqref{difs1} and \eqref{t0t1}, we obtain 
\begin{equation}\label{t22}
\ell(p-t_1(h_0)e) = v(p-t_1(h_0)e) = o(t_0).
\end{equation}
By \eqref{t21}, \eqref{t22}, and the balance of \(S^c_{h_0}[v](y_{t_0})\) with respect to $y_{t_0}$,
we have 
	\[ \ell(y_{t_0}) \leq C_n(2at_0 + o(t_0)), \] contradicting the fact that \(\ell(y_{t_0}) \geq \ell(y_{t_0}) - v(y_{t_0}) = h_0\), provided \(t_0 \ll h_0\) was chosen initially.
Hence, the claim is proved.

We then \emph{claim} that \(\liminf_{h\rightarrow 0}t_2(h)\leq 0\), provided $t_0$ is chosen sufficiently small initially. 
Suppose, to the contrary, that there exists a constant \(a_0 \in (0, t_0)\) and a sequence \(\{h_j\}\) converging to \(0\), such that \(t_2(h_j) \geq a_0\) for all \(j\). Hence, from \eqref{difs1},
\begin{equation}\label{t2hj}
v(p+t_2(h_j)e)\geq \frac{1}{2}aa_0,
\end{equation}
provided \(t_0\) is small enough.
Moreover, since $S^c_{h_j}[v](y_{t_0})$ is centred at $y_{t_0}$ we have the segment 
\begin{equation}
I=\left\{p+\theta e: -t_0-\frac{1}{C_n}t_0<\theta<a_0\right\} \subset S^c_{h_j}[v](y_{t_0})\quad\forall\,j=1,2,\cdots.
\end{equation}
Suppose \(S^c_{h_j}[v](y_{t_0})=\{y\in\mathbb{R}^n: v(y)<v(y_{t_0})+x_j\cdot (y-y_{t_0})+h_j\}\) for some vector \(x_j\in \mathbb{R}^n\).
Since \(v\) is globally Lipschitz, \( \sup_j |x_j| \leq \|v\|_{\text{Lip}} < \infty \). Consequently, by taking a subsequence if necessary, we may assume \(x_j \to \bar{x}\) as \(j \to \infty\).
In the limit, we thus obtain 
	\[ 0 \leq v(y) \leq \bar{\ell}(y) \quad \forall\, y \in I, \] 
where \( \bar{\ell}(y) = v(y_{t_0}) + \bar{x} \cdot (y - y_{t_0}) \). 
Since \(y_{t_0} \in I\), \(v(y_{t_0}) = \bar{\ell}(y_{t_0})\) and \(v \leq \bar{\ell}\) on \(I\), the convexity of \(v\) implies that \(v \equiv \bar{\ell}\) on \(I\), namely \(v(p + te)\) is affine with respect to \(t\) for \(t \in (-t_0 - \frac{1}{C_n}t_0, a_0)\), which contradicts \eqref{difs1}. 
Hence, the claim $\liminf_{h \to 0} t_2(h) \leq 0$ is proved.

Now, choose a $t_0>0$ small enough such that the above two claims hold. By these claims and the continuity of \(t_2(h)\) with respect to \(h\), there exists an \(h'\in(0, h_0)\) such that 
\begin{equation}
t_2(h') = \frac{1}{M}t_0,
\end{equation}
where \(M > 1\) is a large constant to be determined.
By \eqref{t0t1}, we then have 
\begin{equation}
t_0 < t_1(h') < C_n(t_0 + t_2(h')) \leq 2C_nt_0.
\end{equation}
Assume that \(S^c_{h'}[v](y_{t_0}) = \{v < \ell\}\) for some affine function \(\ell\) satisfying \(\ell(y_{t_0}) = v(y_{t_0}) + h'\).

Evaluating the affine function \(\ell\) at specific points yields:
\begin{align}
\ell(p - t_1e) = v(p - t_1(h')e) &= o(t_1) = o(t_0), \\
\ell\left(p + t_2(h')e\right) = v\left(p + \frac{1}{M}t_0e\right) &= \frac{a}{M}t_0 + o(t_0).
\end{align}
The above equations imply that \(\ell\) is increasing along the \(e\) direction (provided $t_0$ is chosen sufficiently small at the beginning), leading to:
\begin{equation}\label{makec1}
\ell(p) - v(p) = \ell(p) \geq \ell(y_{t_0}) \geq \ell(y_{t_0}) - v(y_{t_0}) \approx \left| \inf_{S^c_{h'}[v](y_{t_0})} (v - \ell) \right|.
\end{equation}

Applying John's lemma, we identify an affine transformation \(L\) such that \(L(y_{t_0}) = 0\), and 
\begin{equation}\label{normdd}
B_1(0) \subset L\left(S^c_{h'}[v](y_{t_0})\right) \subset B_{C_n}(0).
\end{equation}
Since the affine transformation preserves the ratio of distances, we have
	\[ \frac{\text{dist}\left(L(p), \partial L\left(S^c_{h'}[v](y_{t_0})\right)\right) }{\text{diam}(L\left(S^c_{h'}[v](y_{t_0})\right))} = \frac{\text{dist}(p,\partial S^c_{h'}[v](y_{t_0})}{\text{diam}(S^c_{h'}[v](y_{t_0})} \leq \frac{t_2(h')}{t_1(h')} \leq \frac{1}{M}.\]
Hence, by \eqref{normdd}
\begin{equation}\label{makec2}
\text{dist}\left(L(p), \partial L\left(S^c_{h'}[v](y_{t_0})\right)\right)   \leq \frac{C_n}{M}.
\end{equation}

Since \(y_{t_0} \in \Omega_1^* \cap B_{\delta_0}(p)\) and \(h' \leq h_0\), Lemma \ref{locc2} implies \(S^c_{h'}[v](y_{t_0}) \subset B_{2r_0}(p)\), ensuring that \(S^c_{h'}[v](y_{t_0}) \cap \Omega_2^* = \emptyset\). 
The convexity of $\Omega^*_1$ leads to the volume comparison:
\begin{equation}\label{makec3}
\left| \left(\frac{1}{2}S^c_{h'}[v](y_{t_0})\right)_{y_{t_0}} \cap \Omega^* \right| \geq C_n \left| S^c_{h'}[v](y_{t_0}) \cap \Omega^* \right|.
\end{equation}
Then by the estimate \eqref{la1}, we can obtain
\begin{equation}\label{makec4}
 \left|\inf_{S^c_{h'}[v](y_{t_0})} (v-\ell)\right|^n \geq C_1 C_n |S^c_{h'}[v](y_{t_0}) \cap \Omega^*| |S^c_{h'}[v](y_{t_0})|,
\end{equation}
where \(C_1\) and \(C_n\) are constants independent of \( M \).

On the other hand, by estimates \eqref{ua1}, \eqref{makec1} and \eqref{makec2}, we can deduce
\begin{equation}
 \left|\inf_{S^c_{h'}[v](y_{t_0})} (v-\ell)\right|^n \lesssim |\ell(p)-v(p)|^n \leq \frac{C_2 C_n}{M} |S^c_{h'}[v](y_{t_0}) \cap \Omega^*| |S^c_{h'}[v](y_{t_0})|,
\end{equation}
where \(C_2\) is another constant  independent of \( M \).
However, this contradicts the estimate \eqref{makec4}, when the constant \(M\) is sufficiently large. 

Consequently, we conclude that the function \(v\) is differentiable at the point \(p\). By selecting a sufficiently small \(\bar r_1 < \delta_0\), we further obtain:
\begin{equation}\label{newextend1}
\partial^{-}v(B_{\bar r_1}(p)) \subset B_{\rho_0}(0) \subset \Omega.
\end{equation}
Last, for any point \(y \in \overline{\Omega_1^*} \cap B_{\bar r_1}(p)\), we can employ the previous argumentation for establishing the differentiability of \(v\) at \(p\) to similarly show that \(v\) is differentiable at \(y\).
\end{proof}

\begin{remark}
\emph{ Let \( r_0, \delta_0, \rho_0, \) and \( \bar r_1 \) be the constants given in preceding lemmas. 
According to \eqref{newextend1}, for any point \( y \in \Omega_1^* \cap B_{\bar r_1}(p) \), the set \( \partial^{-}v(y) \cap \Omega \) is non-empty. By invoking Proposition 2 from \cite{FK1} and the subsequent discussions, it follows that the function \( v \) is continuously differentiable and exhibits strict convexity within \( \Omega_1^* \cap B_{\bar r_1}(p) \). }
\end{remark}

\begin{lemma}[Localisation]
The function $v$ satisfies the Monge-Amp\`ere inequality:
\begin{equation}\label{lockey1}
\frac{1}{\lambda^2}\chi_{_{\Omega^*_1}} \leq \det\, D^2v \leq \lambda^2\chi_{_{\Omega^*_1}} 
\quad \text{in} \ B_{\bar r_1}(p)
\end{equation}
in the Alexandrov sense, as in \eqref{aama}. 
\end{lemma}

\begin{proof}
It suffices to show $\det D^2v$ has no mass in $B_{\bar r_1}(p) \setminus \Omega_1^*$, namely $|\partial^{-}v(B_{\bar r_1}(p) \setminus \Omega_1^*)| =0$. 

In fact, let $ y\in B_{\bar r_1}(p) \setminus \Omega_1^*$, and $x\in \partial^{-}v(y)$. Owing to \eqref{newextend1}, it is guaranteed that \( x \) belongs to \( \Omega = \Omega_1 \cup \Omega_2 \cup \mathcal{F} \).
Recall that $u_i$ is $C^1$ and strictly convex inside $\Omega_i$ for $i=1, 2.$
 If \( x \) were to lie in either of \( \Omega_i \), it would imply that \( y \in\Omega_i^* \), being the gradient \( Du_i(x) \). This, however, contradicts our initial assumption that \( y \in B_{\bar r_1}(p) \setminus \Omega_1^* \). Consequently, it must be that \( x \) resides in the set \( \mathcal{F} \).

This reasoning concludes that the subdifferential \( \partial^{-}v \) maps \( B_{\bar r_1}(p) \setminus \Omega_1^* \) entirely into \( \mathcal{F} \). 
 By \cite[Theorem 5.1]{CL1}, we have that $\mathcal{F}$ is a Lipschitz hypersurface, and hence its Lebesgue measure is zero.
Therefore, the Lebesgue measure of the image set satisfies $|\partial^{-}v(B_{\bar r_1}(p) \setminus \Omega_1^*)| =0$.
\end{proof}

Utilising Lemmas \ref{locc1} and \ref{locc2}, we are able to determine constants \( \bar r_0 < \bar r_1 \) and \( h_0 > 0 \) such that:
\begin{equation}\label{lockey2}
S_h^c[v](y) \subset B_{\bar r_1}(p)  \quad \forall\, y \in B_{\bar r_0}(p) \cap \Omega_1^*, \ \  \forall\, h \leq h_0.
\end{equation}
With these preparations, we are now ready to prove Theorem \ref{t111} as follows.
\begin{proof}[Proof of Theorem \ref{t111}]
By \eqref{lockey1} and \eqref{lockey2}, we can invoke the proof of \cite[Lemma 2.2, Corollary 2.3]{C96} (see also \cite[Theorem 7.13]{CM}) to obtain a quantitative uniform convexity of $v$ within $B_{\bar{r}_0}(p)\cap \overline{\Omega_1^*}$. Specifically, this implies that for any two points $y, \tilde{y} \in B_{\bar{r}_0}(p) \cap \overline{\Omega_1^*}$, one has the inequality 
	$$ |Dv(y) - Dv(\tilde{y})| \geq C|y - \tilde{y}|^a,$$ 
where $a > 2$ and $C$ is a positive constant.

Since $u_1 = v^*$ in $\Omega_1$, where $v^*$ is the Legendre transform of $v$, a basic and well known property of the Legendre transform \cite[Remark 7.10]{CM} allows us to deduce that $u_1 \in C^{1,\beta}(B_r(0) \cap \overline{\Omega_1})$ for some $\beta \in (0, 1)$ and $r > 0$. 
By a similar argument, and a possibly smaller $r$, we can ensure that $u_2 \in C^{1,\beta}(B_r(0) \cap \overline{\Omega_2})$.

Inferred from \eqref{gn1}, we conclude that the singular set $\mathcal{F}$ exhibits $C^{1,\beta}$ regularity in $B_r(0)$. Finally, a standard covering argument ensures the $C^{1,\beta}$ regularity of $\mathcal{F}$ inside $\Omega$.
\end{proof}

\begin{remark}
\emph{
Since \( u_i \in C^{1,\beta}(B_r(0) \cap \overline{\Omega_i}) \) for \( i = 1, 2 \), we have \( p = y_0 = Du_1(0) \) and \( \hat{p} = \hat{y}_0 = Du_2(0)\).
}
\end{remark}

\section{Obliqueness}\label{S4}

In this section, we further assume that $\Omega^*_i$ are $C^2$ and uniformly convex domains, for $i = 1, 2$, and that the functions $f, g$ are positive and belong to $C^\alpha_{\text{loc}}(\Omega)$ and $C^\alpha(\overline{\Omega^*})$, respectively, for some $\alpha\in(0,1)$. Without loss of generality, we can translate our coordinate system such that $x_0 \in \mathcal{F}$ is at the origin, i.e., $x_0 = 0$.

By subtracting a constant, we may assume \( u_1(0) = u_2(0) = u(0) = 0. \)
Denote $y_{_{01}} = y_0 = Du_1(0)\in  \partial\Omega^*_1$ and $y_{_{02}} = \hat{y}_0 = Du_2(0)\in  \partial\Omega^*_2$.
As in \eqref{lockey1} and \eqref{lockey2}, we identify constants \(\bar{r}_1 > 0\), \(\bar{r}_0 > 0\), and \(h_0 > 0\) such that the following holds:
\begin{equation}\label{lockey3}
\det\, D^2v = \tilde g_i \quad \text{ in } B_{\bar r_1}(y_{_{0i}}), \quad i = 1,2,
\end{equation}
interpreted in the Alexandrov sense, where $\tilde{g}_i(y) = \frac{g(y)}{f(Dv(y))}$ for $y \in \overline{\Omega_i^*} \cap B_{\bar r_1}(y_{_{0i}})$, and $\tilde{g}_i(y) = 0$ for $y \in B_{\bar r_1}(y_{_{0i}}) \setminus \overline{\Omega_i^*}$, with $i = 1, 2$. Furthermore, for each $i=1,2$, we have:
\begin{equation}\label{lockey4}
S_h^c[v](y) \subset B_{\bar r_1}(y_{_{0i}})  \quad \forall\, y \in B_{\bar r_0}(y_{_{0i}}) \cap \Omega_i^*, \quad \forall\, h \leq h_0.
\end{equation} 

The subsequent analysis for $v$ is in a neighbourhood of $y_0$, which nevertheless also applies to $v$ near $\hat y_0$ analogously. 
Hereafter, for brevity we
denote $S^c_{h}[v](y_0)$ by  $S^c_{h}[v]$. 
In light of \eqref{lockey3} and \eqref{lockey4}, together with the convexity of $\Omega_1^*$, we infer that the Monge-Amp\`ere measure $\det\,D^2v$ satisfies a doubling property in $S_h^c[v](y)$ for all $y\in \overline{\Omega_1^*} \cap B_{\bar r_1}(y_{0})$ and $h < h_0$. 
From \cite[Lemma 2.2]{C96}, the sections $S_h^c[v]$ are of geometric decay. 
Specifically, given any $0 < s_1 < \bar{s} < 1$, there exists a constant $s_0 \in (0,1)$ independent of $h$ such that $\forall$ $y, \bar{y} \in \overline{\Omega_1^*} \cap B_{\bar{r}_1}(y_{0})$ with $y \in s_1 S^c_h[v](\bar{y})$, one has
\begin{equation}\label{geodc}
S^c_{s h}[v](y) \subset \bar{s} S^c_h[v](\bar{y}) \quad \forall\, h \in (0, h_0),\ \forall\, s \in (0, s_0).
\end{equation}
Moreover, by \cite[Corollary 2.3]{C96} (see also \cite[Proof of Theorem 1]{C92b}), the function $v$ satisfies a strict convexity estimate of the form:
\begin{equation}\label{strictconvex}
v(y) \geq v(y_0) + Dv(y_0) \cdot (y - y_0) + C |y - y_0|^{1 + a}  \quad \forall\, y \in \overline{\Omega_1^*} \ \text{ near } y_0,
\end{equation}
where $a > 1$ and $C$ are constants.

Similar to that in the optimal partial transport \cite[Lemma 2.6]{CLW3}, we have the uniform density property as follows:
\begin{lemma}[Uniform Density]\label{ud1}
There exists a constant $\delta > 0$, depending only on $\Omega, \Omega^*, \lambda$, and $\text{dist}(x_0, \partial\Omega)$, such that
\[
\frac{|S_h^c[v] \cap \Omega^*_1|}{|S_h^c[v]|} \geq \delta \quad \ \forall\, h\leq h_0.
\]
\end{lemma}

The proof of Lemma \ref{ud1} aligns closely with that in \cite[Theorem 3.1]{C96} and \cite[Lemma 2.6]{CLW3}, which is omitted here for brevity.

In the following, the notation $a \approx b$ means that $C^{-1}a \leq b \leq Ca$ for some universal positive constant $C.$  Given a convex domain $D \subset \mathbb{R}^n$, we say that $D$ has a \emph{good shape} if $B_{C^{-1}}(x^*) \subset D \subset B_{C}(x^*)$, where $C$ is a universal positive constant and $x^*$ is the centre of mass of $D$.
By a change of coordinates, assume the hyperplane \( \mathcal{H}:=\{y \in \mathbb{R}^n : (y-y_0)\cdot e_1 = 0\} \) tangentially touches \( \partial\Omega_1^* \) at \( y_0 \). 
At this point we may apply \cite[Corollary 2.7]{CLW3} for which the key hypothesis is the uniform density estimate of Lemma \ref{ud1}. From this, we can deduce both a volume estimate and a tangential \( C^{1,1-\epsilon} \) estimate for the function \( v \).

\begin{corollary}\label{co21}
The following assertions hold true:
\begin{enumerate}
\item[(i)] \emph{Volume Estimate:}
\begin{equation}\label{vol}
|S_h[v](y_0)\cap \Omega_1^*| \approx |S_h^c[v](y_0)\cap \Omega_1^*| \approx |S_h^c[v](y_0)| \approx h^{\frac{n}{2}}.
\end{equation}
Moreover, for any affine transform $\mathcal{A}$, 
if either $\mathcal{A}(S_h^c[v](y_0))$ or $\mathcal{A}(S_h[v](y_0) \cap \Omega_1^*)$ has a good shape, so is the other one.

\item[(ii)] \emph{Tangential $C^{1,1-\epsilon}$ regularity of $v$:} $\forall$ $\epsilon > 0$, $\exists$ a constant $C_\epsilon$ such that:
\begin{equation}\label{tanc2}
B_{C_\epsilon h^{\frac{1}{2}+\epsilon}}(y_0) \cap \mathcal{H} \subset S_h^c[v](y_0) \quad \forall\, h > 0\text{ sufficiently small}.
\end{equation}
\end{enumerate}
\end{corollary}

\begin{remark}\label{reeq1}
\emph{ By the strict convexity of $v$ (see \eqref{strictconvex}) and the fact that $S^c_h[v]$ is centred at $y_0$,
we have an equivalence relation between $S_h[v]$ and $S^c_{h}[v]$, namely $\forall\,h>0$ small, 
\begin{equation}\label{equi0}
S^c_{b^{-1}h}[v] \cap \Omega_1^*\subset S_h[v] \cap \Omega_1^* \subset S^c_{bh}[v] \cap \Omega_1^*,
\end{equation}
where \( b>1 \) is a constant independent of \( h \). 
The reader is referred to \cite[Lemma 2.2]{CLW1} for a proof. }
\end{remark}

Denote \( \nu = \nu(0) \) and \( \hat\nu = \hat\nu(0) \) as the unit inner normals to \( \partial\Omega_1 \) and \( \partial\Omega_2 \) at \( 0\in\mathcal{F} \), respectively. 
Given that \( \mathcal{F} \) is  \( C^{1,\beta} \) regular, one has \( \nu = -\hat\nu \).
In a similar vein, let \( \nu^* = \nu^*(y_0) \) and \( \hat\nu^* = \hat\nu^*(\hat y_0) \) denote the unit inner normals to \( \partial\Omega^*_1 \) and \( \partial\Omega^*_2 \) at \( y_0 \) and \( \hat y_0 \), respectively.

\begin{proposition}\label{oblique2}
The inner product of the normals, \( \nu \cdot \nu^* \) and \( \hat\nu \cdot \hat\nu^* \), are both strictly positive, i.e., \( \nu \cdot \nu^* > 0 \) and \( \hat\nu \cdot \hat\nu^* > 0 \).
\end{proposition}

By a change of coordinates, we can assume $\nu=e_n$. Then from \eqref{gn1}, it follows that 
	\[ e_n = \frac{y_0 - \hat y_0}{|y_0 - \hat y_0|}. \]
By subtracting a linear function from $u$, we can assume \( y_0 = l e_n \), \( \hat y_0 = -l e_n \) for a constant \( l \geq \frac{1}{2}\text{dist}(\Omega_1^*, \Omega_2^*)>0 \).
The proof of Proposition \ref{oblique2} proceeds by contradiction and will occupy most of the remainder of the paper. In particular, we need to address three distinct cases, which are addressed in Sections 4.1, 4.2, and 4.3, respectively.  

Without loss of generality, suppose the obliqueness conditions are not satisfied at $y_0\in \partial\Omega^*_1$, namely $e_n\cdot \nu^*=0$.
By a rotation of coordinates, we may assume that \( \nu^* = e_1 \).

Given that \( \mathcal{F} \in C^{1,\beta} \), locally near $0$, \( \mathcal{F} \) can be expressed as 
\begin{equation}\label{calf1}
 \mathcal{F} = \{x\in\mathbb{R}^n : x_n = \rho(x_1, \ldots, x_{n-1})\}
 \end{equation}
for a function \( \rho \in C^{1,\beta} \) satisfying $\rho(0)=0$ and $D\rho(0)=0$. 

Given that \( \partial\Omega^*_1 \in C^2 \) is uniformly convex, locally near \( y_0 \), \( \partial\Omega^*_1 \) can be described as
  \begin{equation}\label{ostar1}
  \partial\Omega^*_1 = \{y\in\mathbb{R}^n : y_1 = \rho^*(y_2, \ldots, y_n-l)\}
  \end{equation}
for a uniformly convex function \( \rho^* \in C^2 \) satisfying $\rho^*(0)=0$ and $D\rho^*(0)=0$.

Considering the convexity of \( v \) and the condition \( Dv(y_0) = 0 \), we may further adjust \( v \) by adding a constant to ensure \( v(y_0) = 0 \) and \( v \geq 0 \).
To proceed, we divide our analysis into three distinct cases:

\begin{description}
    \item[Case I]  $\nu\cdot \nu^*=0,\ \hat\nu \cdot \hat\nu^* = 0$, and $\nu, \nu^*, \hat{\nu}^*$ are non-coplanar.
    \item[Case II] $\nu\cdot \nu^*=0,\ \hat\nu \cdot \hat\nu^* = 0$, and $\nu, \nu^*, \hat\nu^*$ are coplanar.
 \item[Case III]  $\nu\cdot \nu^*=0,\ \hat\nu \cdot \hat\nu^* > 0$.
\end{description}

\subsection{Case I:  $\nu\cdot \nu^*=0,\ \hat\nu \cdot \hat\nu^* = 0$, and $\nu, \nu^*, \hat{\nu}^*$ are non-coplanar.}

Since \( \nu \), \( \nu^* \), and \( \hat{\nu}^* \) are non-coplanar, there exists an affine transformation \( A \) with \( \det\,A = 1 \), such that \( A\nu^* \cdot A\hat{\nu}^* =0 \). Note that \( u(A^{-1}x) \) is the potential function for the optimal transport from \( A\Omega \) to \( (A^t)^{-1}\Omega^* \), 
\begin{itemize}
\item \( \frac{(A^t)^{-1}\nu}{|(A^t)^{-1}\nu|} \) is the unit inner normal of \( A\Omega_1 \) at \( 0 \), 
\item \( \frac{A\nu^*}{|A\nu^*|} \) is the unit inner normal of \( (A^t)^{-1}\Omega_1^* \) at \( (A^t)^{-1}y_0 \), 
\item and \( \frac{A\hat{\nu}^*}{|A\hat{\nu}^*|} \) is the unit inner normal of \( (A^t)^{-1}\Omega_2^* \) at \(  (A^t)^{-1}\hat{y}_0 \). 
\end{itemize}
One can verify that \( \frac{(A^t)^{-1}\nu}{|(A^t)^{-1}\nu|} \), \( \frac{A\nu^*}{|A\nu^*|} \), and \( \frac{A\hat{\nu}^*}{|A\hat{\nu}^*|} \) are mutually perpendicular to each other.

Hence, under an appropriate affine transformation and a rotation of coordinates, we can assume that \( y_0 \) and \( \hat{y}_0 \) are as before, and \( \mathcal{F} \), \( \partial\Omega^*_1 \) are as in \eqref{calf1} and \eqref{ostar1} respectively, and locally near \( \hat{y}_0 \), the boundary of \( \Omega_2^* \)  is represented as
\[
\partial\Omega_2^* = \left\{ y\in\mathbb{R}^n : y_2 = \rho_2^*(y_1, y_3, \ldots, y_{n-1}, y_n + l) \right\}
\]
for a uniformly convex function \( \rho_2^* \in C^2 \) satisfying \( \rho_2^*(0) = 0 \) and \( D\rho_2^*(0) = 0 \).

Recall the sub-level set $S_h[v]=S_h[v](y_0)=\{v<h\}$ in \eqref{sub}.
Define the width function
\[
d_e := \sup \left\{ |(y-y_0) \cdot e| : y \in S_h[v] \cap \Omega_1^* \right\}\quad \mbox{ for } e\in\mathbb{S}^{n-1}. 
\]
Let $p_e \in \overline{S_h[v] \cap \Omega_1^*}$ be a point where $|(p_e-y_0) \cdot e| = d_e$. When specifically considering the unit vector $e_1$, we shall denote $p_{e_1}$ simply as $p=(p_1, \cdots, p_n)$.

\begin{lemma}\label{ide1}
For any $\epsilon>0$ small, there exists a constant $C_\epsilon$ such that 
	\[ d_{e_1} \leq C_{\epsilon}h^{\frac{2}{3} - \epsilon}. \]
Moreover, for any unit vector $e \in \mathrm{span}\{e_2, \ldots, e_n\}$,  
	\[ d_e^2 \leq C d_{e_1} \] 
for some universal constant $C$. 
\end{lemma}

\begin{proof}
Let \( p \) be as above. Consider \( q \), the point of intersection between the segment \( \overline{op} \) and \( \partial\Omega_1^* \). 
Since both \( p \) and the origin are in \( \overline{S_h[v]} \), by convexity one has \( q \in \overline{S_h[v]} \). Thus,
\[ q \in \overline{S_h[v]} \cap \partial\Omega_1^*. \]
Recall that \( d_{e_1} = p_1 \). 
As the intersection \( S_h[v] \cap \Omega_1^* \) converges to \( y_0 \) as \( h\rightarrow 0 \), we deduce that \( \frac{q_n}{p_n} \geq \frac{1}{2} \) for sufficiently small \( h \). Consequently,
\[ q_1 \geq \frac{q_n}{p_n} p_1 \geq \frac{1}{2} d_{e_1}. \] 

Define \( e_q := \frac{q - y_0}{|q - y_0|} \). 
Let \( D \subset \mathrm{span}\{e_q, e_1\} \) be the planar region bounded by \( \partial\Omega_1^* \cap \mathrm{span}\{e_q, e_1\} \) and the segment \( \overline{y_0q} \). 
Since \( \partial\Omega_1^*\in C^2 \) is uniformly convex, one has 
	\[ \mathcal{H}^2(D) \geq C|q_1|^{\frac{3}{2}} \geq Cd_{e_1}^{3/2} \] 
for a constant \( C \) independent of \( h \), where \( \mathcal{H}^2(\cdot) \) denotes the 2-dimensional Hausdorff measure, see Figure 4.1.

Select \( \{\tilde{e}_2, \ldots, \tilde{e}_{n-1}\} \) to be an orthonormal basis for the orthogonal complement of \( \mathrm{span}\{e_1, e_p\} \) in \( \mathbb{R}^n \). The tangential \( C^{1, 1-\epsilon} \) estimate of \( v \) asserts that
\[ y_0 + C_\epsilon h^{\frac{1}{2}+\epsilon}\tilde{e}_i \in S^c_{bh}[v] \quad \text{for} \quad i=2, \ldots, n-1, \]
where \( b \) is the constant in \eqref{equi0}.
Define \( G \) to be the convex hull of \( D \) and points \(\{ y_0 + C_\epsilon h^{\frac{1}{2}+\epsilon}\tilde{e}_i, \, i=2, \ldots, n-1 \}\). By convexity, it follows that \( G \subset S^c_{bh}[v] \).

Combining these results and the inclusion $S_h[v]\cap \Omega_1^* \subset S^c_{bh}[v]$, we obtain:
\[
C_\epsilon h^{(\frac{1}{2}+\epsilon)(n-2)}\mathcal{H}^2(D) \leq |G| \leq |S^c_{bh}[v]| \approx h^{\frac{n}{2}},
\]
which in turn implies $|d_{e_1}| \leq C_{\epsilon}h^{\frac{2}{3} - \epsilon}$.

Fix any unit vector $e \in \mathrm{span}\{e_2, \ldots, e_n\}$. The uniform convexity of $\partial\Omega_1^*$ yields 
	\[ d_{e_1} \geq (p_e-y_0) \cdot e_1 \geq C_1((p_e-y_0) \cdot e)^2 = C_1d_e^2  \]
for some universal constant $C_1>0$, which implies $d_e^2 \leq C d_{e_1}$ for a universal constant $C$.
\end{proof}

\begin{corollary}\label{decayrate1}
For all $y \in \Omega_1^*$, the inequality $v(y) \geq C_\epsilon |y - y_0|^{3 + \epsilon}$ holds.
\end{corollary}
\begin{proof}
From Lemma \ref{ide1}, we deduce that 
	\[ v(y) \geq h \quad \forall\, y \in \Omega_1^*\  \mbox{ with } |y-y_0| \geq C_\epsilon h^{\frac{1}{3} - \epsilon}. \]
Consequently, this implies \( v(y) \geq C_\epsilon |y - y_0|^{3 + \epsilon} \) for all \( y \in \Omega_1^* \).
\end{proof}

\begin{remark}
\emph{ Analogously, one has $v(y) \geq C_\epsilon |y - \hat{y}_0|^{3 + \epsilon}$ holds for all $y \in \Omega_2^*$. }
\end{remark}

Recalling \eqref{calf1}, we have the following estimate for $\rho$:
\begin{lemma}\label{decayrate2}
The function $\rho$ satisfies $|\rho(x')| \leq C_\epsilon |x'|^{\frac{3}{2} - \epsilon}$, $x'=(x_1,\cdots,x_{n-1})$.
\end{lemma}

\begin{proof}
Recall that \( Du_1(0)=y_0=l e_n \) and \( u_1(0)=0 \). For any point \( x \in \mathbb{R}^n \), one has:
\begin{align*}
0 &\leq u_1(x) - l x_n \\
  &= \sup_{y \in \Omega_1^*} \left( x \cdot (y - y_0) - v(y) \right) \\
  &\leq \sup_{y \in \Omega_1^*} \left( x \cdot (y - y_0) - C_\epsilon |y - y_0|^{3 + \epsilon} \right) \\
  &\leq \sup_{y \in \mathbb{R}^n} \left( x \cdot y - C_\epsilon |y|^{3 + \epsilon} \right) \\
  &\leq C_\epsilon |x|^{\frac{3}{2} - \epsilon}.
\end{align*}
Similarly, \( 0 \leq u_2(x) + lx_n \leq C_\epsilon |x|^{\frac{3}{2} - \epsilon} \).
By the convexity of \( u_1 \) and \( u_2 \),
we have:
\begin{equation}\label{condes}
|D(u_1(x) - l x_n)| \leq C_\epsilon |x|^{\frac{1}{2} - \epsilon}, \quad |D(u_2(x) + lx_n)| \leq C_\epsilon |x|^{\frac{1}{2} - \epsilon}.
\end{equation}

Since the singular set $\mathcal{F}$ is $C^{1,\beta}$ regular, we have
\begin{equation}
|\rho(x')| \leq C |x'|^{1 + \beta}.
\end{equation}
This, together with \eqref{condes}, leads to:
\begin{equation}\label{pest1}
|D_{x'}u_i(x)| \leq C_\epsilon |x'|^{\frac{1}{2} - \epsilon}, \quad \text{for } i = 1, 2, \text{ and } x \in \mathcal{F}.
\end{equation}

Now, at any point $x \in \mathcal{F}$, the unit inner normal vector $\nu(x)$ of $\Omega_1$ is explicitly given by:
\begin{equation}
\nu(x) = \frac{(-D_{x'}\rho(x'), 1)}{\sqrt{|D_{x'}\rho(x')|^2 + 1}} = \frac{\left(D_{x'}(u_1(x) - u_2(x)), \partial_{x_n}(u_1(x) - u_2(x))\right)}{|Du_1(x) - Du_2(x)|}.
\end{equation}
By \eqref{pest1}, we then obtain
\begin{equation}
|D_{x'}\rho(x')| \leq C_\epsilon |x'|^{\frac{1}{2} - \epsilon},
\end{equation}
which implies the desired estimate by integration. 
\end{proof}

Next, we prove an ``\emph{above the tangent}'' property, namely \(\mathcal{F}\cap\text{span}\{\nu, \nu^*\}\) is  above its tangent plane at \(0.\) This property is  used in the proof 
of Lemma \ref{gdspl1}, see the argument above \eqref{gp22}.

\begin{lemma}\label{saddle}
There exists a small $t_0 > 0$ such that
\begin{equation}\label{saddle2}
\rho(t, 0, \ldots, 0) > 0 \quad \text{whenever } 0 < |t| < t_0.
\end{equation}
\end{lemma}

\begin{proof}
Suppose to the contrary that for any small $t_0>0$, there exists a non-zero $t$ satisfying $|t|<t_0$ such that $t e_1 \in \overline{\Omega_1}$. Define $q := Du_1(te_1) \in \overline{\Omega_1^*}$ and observe that $Dv(q) = te_1$.

If $t < 0$, the uniform convexity of $\Omega_1^*$ implies that
\[(Dv(q) - 0) \cdot (q - y_0) < 0,\]
which contradicts the convexity of $v$. Consequently, one must have $t > 0$.

Now, set $h = v(q)$. Lemma \ref{ide1} implies that $q_1 \leq C_\epsilon h^{\frac{2}{3} - \epsilon}$. Since $Dv(q) = te_1$, the vector $e_1$ is the unit outward normal to the level set $S_h[v]$ at the point $q$. This leads to the inclusion
\begin{equation}\label{saddle1}
S_h[v] \subseteq \left\{ y \in \mathbb{R}^n : y_1 \leq C_\epsilon h^{\frac{2}{3} - \epsilon} \right\}.
\end{equation}

On the other hand, by the tangential $C^{1,1-\epsilon}$ regularity of $v$ at $\hat{y}_0$, one has
\[ p = \hat{y}_0 + C_\epsilon h^{\frac{1}{2} + \epsilon} e_1 \in S^c_{\frac{1}{2}b^{-1}h}[v](\hat{y}_0) \subseteq S_{\frac{h}{2}}[v], \]
where the last inclusion follows from \eqref{equi0}.
Thus, $p \in S_h[v]$ with
\[ p_1 = C_\epsilon h^{\frac{1}{2} + \epsilon} \gg C_\epsilon h^{\frac{2}{3} - \epsilon} \]
for $h>0$ sufficiently small. However, this contradicts \eqref{saddle1} when $h$ (and consequently $t$) is sufficiently small. 
Hence, the lemma is proved. 
\end{proof}

\begin{figure}[h]\label{f4.1}
 \centering
 \includegraphics[scale=0.33]{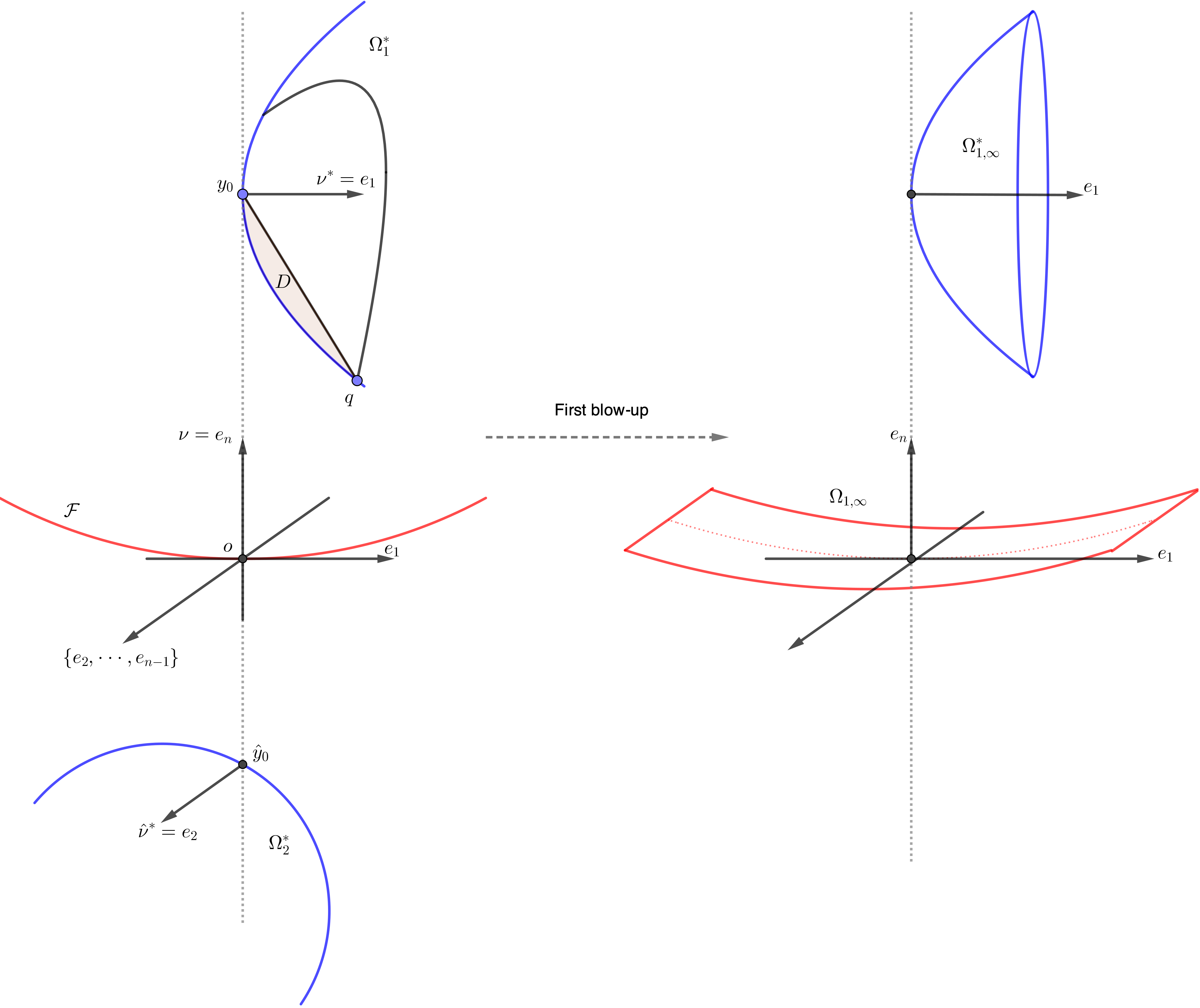}
 \caption{The limit profile after the first blow-up. }
\end{figure}

{\bf First blow-up.}
By a translation of coordinates, we may assume that \( y_0 = 0 \). 
Suppose \( S_h^c[v] \approx E \) (namely \( E \subset S_h^c[v] \subset C_nE \)) for an ellipsoid \( E \) centred at $0$. 
Then, \( E \cap \{y_1 = 0\} \) is an \( (n-1) \)-dimensional ellipsoid aligned along the principal directions \( \bar{e}_2, \ldots, \bar{e}_n \). 
Note that \( \text{span}\{\bar{e}_2, \ldots, \bar{e}_n\} = \text{span}\{e_2, \ldots, e_n\} \).

The ellipsoid $E$ can be expressed as
\begin{equation}\label{elliptan}
E = \left\{y = y_1e_1 + \sum_{i=2}^n \bar{y}_i \bar{e}_i : \frac{y_1^2}{a_1^2} + \sum_{i=2}^n \frac{(\bar{y}_i - k_i y_1)^2}{a_i^2} \leq 1\right\}
\end{equation}
for some constants $a_1, a_i, k_i,$ where $i=2,\cdots, n.$
Applying Lemma \ref{ide1} and the tangential $C^{1,1-\epsilon}$ estimate of $v$ at $0$, we obtain the following bounds:
\begin{equation}\label{ai000}
0 < a_1 \leq C_\epsilon h^{\frac{2}{3}-\epsilon} \quad \text{and} \quad C_\epsilon h^{\frac{1}{2}+\epsilon} \leq a_i \leq C_\epsilon h^{\frac{1}{3}-\epsilon} \text{ for } i=2, \ldots, n.
\end{equation}

Consider the affine transformations \( T_1 \) and \( T_2 \) defined as follows:
\begin{align}
T_1: x = x_1e_1 + \sum_{i=2}^n \bar{x}_i \bar{e}_i &\mapsto z = x_1e_1 + \sum_{i=2}^n (\bar{x}_i - k_ix_1)\bar{e}_i, \label{afT1}\\
T_2: z = z_1e_1 + \sum_{i=2}^n \bar{z}_i \bar{e}_i &\mapsto y = \frac{z_1}{a_1}e_1 + \sum_{i=2}^n \frac{\bar{z}_i}{a_i} \bar{e}_i. \label{afT2}
\end{align}
Setting \( T = T_2 \circ T_1 \), we have \( T(E) = B_1(0) \), and thus \( T(S_h^c[v]) \approx B_1(0) \).

Define the rescaled function \( v_h(y) := \frac{1}{h}v(T^{-1}y) \).
By \eqref{lockey3}, one can verify that
\begin{equation}\label{vhma1}
\det\, D^2v_h(y) = \frac{1}{h^n(\det\, T)^{2}}\tilde{g}_1(T^{-1}y) \quad \forall\, y \in T(B_{\bar{r}_1}(0)).
\end{equation}
Note that 
\( h^n(\det\,T)^{2} \approx 1 \), due to the relation \( T(S_h^c[v]) \approx B_1(0) \).

The gradient of this function transforms as \( Dv_h(x) = T^*Dv(T^{-1}x) \), where \( T^* = \frac{(T^t)^{-1}}{h} = \frac{1}{h}(T_2^t)^{-1}\circ(T_1^t)^{-1} \). A direct computation reveals the inverses of the transposes of \( T_1 \) and \( T_2 \) as
\begin{align*}
(T_1^t)^{-1}&: x = x_1e_1 + \sum_{i=2}^n \bar{x}_i \bar{e}_i \mapsto z = \left(x_1 + \sum_{i=2}^n k_i \bar{x}_i\right)e_1 + \sum_{i=2}^n \bar{x}_i \bar{e}_i, \\
(T_2^t)^{-1}&: z = z_1e_1 + \sum_{i=2}^n \bar{z}_i \bar{e}_i \mapsto y = a_1 z_1e_1 + \sum_{i=2}^n a_i \bar{z}_i \bar{e}_i.
\end{align*}

Let $\Omega_{h1}^* := T(\Omega^*_{1})$. 
Given that the uniform density property (Lemma \ref{ud1}) remains invariant under affine transformations, 
the centred sections of \( v_h \) satisfy
\begin{equation}\label{udvh}
\frac{|S_{\tilde h}^c[v_h] \cap \Omega^*_{h1}|}{|S_{\tilde h}^c[v_h]|} \geq \delta \qquad \forall\, \tilde h \leq 1. 
\end{equation} 
Particularly, since \( S_{1}^c[v_h] = TS^c_h[v] \approx B_1(0) \), \eqref{udvh} implies that
\begin{equation}\label{ballint}
\frac{|B_{C_n}(0) \cap \Omega^*_{h1}|}{|B_1(0)|} \geq \delta.
\end{equation}
Moreover, since $0\in\partial\Omega^*_{h1}$ and $\Omega^*_{h1}$ is convex, by \eqref{ballint} there exists an open convex cone 
	\[ \mathcal{C}=\{x\in\mathbb{R}^n : x\cdot e > c\}\] 
	for some unit vector \(\ e\in\mathbb{S}^{n-1}\),
where $c>0$ is a constant independent of \( h \), such that
\begin{equation}\label{ccin1}
\mathcal{C} \cap B_1(0) \subset \Omega^*_{h1}
\end{equation}
for all sufficiently small \( h \).

We now proceed to establish the following estimates:
\begin{lemma}\label{vesttoday}
There exist constants $\beta \in (0, 1)$ and $a > 1$ such that 
\begin{equation}\label{alpha2h}
\frac{1}{C}|y|^{1+a} \leq v_h(y) \leq C|y|^{1+\beta} \quad\ \forall\, y \in B_{r_0}(0) \cap \Omega_{h1}^*,
\end{equation}
where $C, r_0 > 0$ are constants independent of $h$.
\end{lemma}
\begin{proof}
Note that $v_h$ satisfies $\frac{1}{C}\chi_{\Omega^*_{h1}} \leq \det D^2v_h \leq C\chi_{\Omega^*_{h1}}$ in $B_{r_0}(0)$ for some universal constants $C$ and $r_0$. 
Observe also that the geometric decay property \eqref{geodc} is invariant under rescaling and affine transformations. 
By the same argument of the proof of \cite[Lemma 5.13]{CLW3}, one can obtain the desired estimates \eqref{alpha2h}.
 For the reader's convenience, we outline the main ideas below and refer to \cite{CLW3} for full details. 

Applying \eqref{geodc} with $y=\bar y$ and $\bar s=\frac12$, there exists a constant $\theta<1$, independent of $h$, such that
	\begin{equation}\label{rev1}
	 	S^c_{\theta h}[v](y) \subset \frac12 S^c_h[v](y).
	\end{equation}
Since \eqref{rev1} is invariant under the rescaling \eqref{vhma1}, by iteration we have
	\begin{equation}\label{rev2}
		S^c_{\theta^k \tilde h}[v_h] \subset \frac{1}{2^k} S^c_{\tilde h}[v_h]
	\end{equation}
for $\tilde h<1$ and $k=1,2,...$. For any $y \in B_{r_0}(0) \cap \Omega_{h1}^*$, let $k$ be the positive integer such that $2^{-k}<|y|\leq 2^{-k+1}$. 
Note that by \eqref{vhma1}, we have $B_{1/C}(0)\subset S^c_1[v_h]\subset B_C(0)$. Then, using \eqref{equi0}, \eqref{rev2} and a direct computation, we obtain $v_h(y)\geq C_1|y|^{\beta'}$, where $C_1=b^{-1}\theta^{\frac{\log(2C)}{\log 2}}$, $\beta'=-\frac{\log\theta}{\log 2}$, and $b$ is the constant from \eqref{equi0}. This gives the first inequality in \eqref{alpha2h}. 

For the second inequality in \eqref{alpha2h}, we again use the fact that \eqref{rev1} is invariant under the rescaling \eqref{vhma1}, independent of $h>0,$ to deduce that the limit of $v_h$ is strictly convex. By a compactness argument, there exists a constant $\delta\in(0,1)$ such that $v_h(\frac12y)\leq \frac12(1-\delta)v_h(y)$ for $y \in B_{r_0}(0) \cap \Omega_{h1}^*$. Iterating this inequality yields $v_h(\frac{1}{2^k}y) \leq \frac{1}{2^k}(1-\delta)^kv_h(y)$.  Hence, there exist constants $\alpha'\in(0,1]$ and $C_2>0$, independent of $h$, such that $v_h(y)\leq C_2|y|^{1+\alpha'}$. 

\end{proof}

Let $T^*=\frac{1}{h}(T^t)^{-1}$.
Denote by $\Omega_{h1}$ the set $T^*\Omega_1$. 
By the above strict convexity estimate of $v_h$ in \eqref{alpha2h}, we have the inclusion 
	$$B_r(0) \cap \Omega_{h1} \subset Dv_h(B_{1}(0) \cap \Omega^*_{h1})$$
for some small constant $r > 0$ independent of $h$. Following a rescaling argument similar to the proof of \cite[Lemma 5.15]{CLW3}, we obtain the following inclusion:
\begin{corollary}\label{goodinclu}
For any $R > 0$ large, there exists a constant $M_R > 0$ independent of $h$ such that 
\begin{equation}\label{ginl}
	B_R(0) \cap \Omega_{h1} \subset Dv_h(B_{M_R}(0) \cap \Omega^*_{h1}) \quad\text{for $h>0$ small.}
\end{equation}
\end{corollary}

Passing to a subsequence \( \{h_j\} \to 0 \), we may assume that \( \{v_{h_j}\} \) converges to \( v_{\infty} \) locally uniformly for some convex function \( v_{\infty} \) defined on \( \mathbb{R}^n \). 
Additionally, up to a further subsequence, we may assume that \( T\Omega_1^* \) converges to some convex set \( \Omega_{1,\infty}^* \subset \{x_1\geq 0\} \) locally uniformly in the Hausdorff distance. 
Let us denote by \( \Omega_{1,\infty} \) the interior of \( \partial^- v_{\infty}(\mathbb{R}^n) \). 
Since \( v_{\infty} \) is a convex function defined on the entire \( \mathbb{R}^n \), it is well-known that \( \Omega_{1,\infty} \) is a convex set. 
By \eqref{vhma1} and the continuity of \( f \) and \( g \) at \( 0 \), and passing to a subsequence, we have that
\begin{equation}\label{vhma2}
\det\, D^2v_{\infty} = c_0 \chi_{_{\Omega_{1,\infty}^*}} \quad \text{on } \ \mathbb{R}^n,
\end{equation}
where $c_0$ is a positive constant.
Define \( u_{\infty}(x) := \sup \{ x \cdot y - v_{\infty}(y) : y \in \mathbb{R}^n \} \) for \( x \in \overline{\Omega_{1,\infty}} \).

Since $\Omega_{1,\infty}^*$ is convex, \eqref{vhma2} implies that the Monge-Amp\`ere measure $\det\, D^2v_\infty$ is doubling for convex sets centred at points in $\overline{\Omega_{1,\infty}^*}$. Consequently, by applying the same proof as in \cite[Lemma 5.17]{CLW3}, we establish the following properties for $u_\infty$ and $v_\infty$:
\begin{lemma}\label{uvreg1}
The function $v_\infty$ is $C^1$ and strictly convex in $\overline{\Omega_{1,\infty}^*}$. Moreover, as a convex function defined on $\mathbb{R}^n$, $v_\infty$ is differentiable at every point $y \in \overline{\Omega_{1,\infty}^*}$. The function $u_\infty$ is $C^1$ and strictly convex in $B_{r_0}(0)\cap \overline{\Omega_{1,\infty}}$ for some small $r_0>0$.
\end{lemma}
\begin{remark}\label{uinfty1}
\emph{
By the definition of $\Omega_{1,\infty}$, one has 
	$$Du_\infty(B_{r_0}(0) \cap \overline{\Omega_{1,\infty}}) \subset\overline{\Omega_{1,\infty}^*} \subset \{x_1 \geq 0\},$$ 
which implies that $u_\infty$ is nondecreasing in the $e_1$-direction near $0.$ }

\emph{ Since \( v_\infty \) is differentiable at \( 0 \), one has
\[
\partial^- v_\infty(B_r(0)) \subset B_{r_0}(0) \cap \overline{\Omega_{1,\infty}}
\]
provided \( r \) is sufficiently small. 
We \emph{claim} that: \( v_\infty \), as a convex function defined on \( \mathbb{R}^n \), is \( C^1 \) in \( B_{r}(0) \) for some \( r > 0 \) sufficiently small. }

\emph{ Indeed, suppose \( v_\infty \) is not differentiable at some \( y \in B_r(0) \), which would imply that \( \partial^- v_\infty(y) \) contains at least two distinct points $x, \bar{x} \in B_{r_0}(0) \cap \overline{\Omega_{1,\infty}}$. 
Since $\Omega_{1,\infty}$ is convex, and $u_\infty$ is the Legendre transform of \( v_\infty \), it follows that \( u_\infty \) is affine along the segment connecting \( x \) and \( \bar{x} \), which contradicts the strict convexity of \( u_\infty \) in \( B_{r_0}(0) \cap \overline{\Omega_{1,\infty}} \).  }
\end{remark}

\begin{lemma}\label{gdspl1}
After an appropriate affine transformation, we obtain (see Figure \ref{f4.1}):
\begin{enumerate}
\item $\Omega_{1,\infty}^* = \{x \in \mathbb{R}^n: x_1 > P(x)\}$, where $P$ is a non-negative homogeneous quadratic polynomial satisfying $P(0) = 0$ and $DP(0) = 0$.
\item $\Omega_{1,\infty} = \{x \in \mathbb{R}^n: x_n > \rho_\infty(x_1)\}$ for some convex function $\rho_{\infty}$ when $|x_1|$ is small. Moreover, $\rho_\infty(0) = 0$ and $\rho_\infty \geq 0$.
\end{enumerate}
\end{lemma}

\begin{proof}
Fix any unit vector $e \in \text{span}\{e_2, \ldots, e_{n-1}\}$. Let $t_h$ be the positive number such that
\[
\left|\frac{1}{h}(T_2^t)^{-1}(t_he)\right| = 1,
\]
yielding
\[
t_h \leq \frac{h}{\min_{2 \leq i \leq n} a_i} \leq C_\epsilon h^{\frac{1}{2} - \epsilon}.
\]
For any $t$ satisfying $|t| < h^{-\epsilon}t_h\leq C_\epsilon h^{\frac{1}{2} - 2\epsilon}$, define $x^t := te = (0, x^t_2, \ldots, x^t_{n-1}, 0)$. It is straightforward to verify that 
\begin{equation}\label{bp112}
\left|T^*(h^{-\epsilon}t_he)\right|\geq \left|\frac{1}{h}(T_2^t)^{-1}(h^{-\epsilon}t_he)\right| \rightarrow \infty\quad \text{as}\ h \rightarrow 0.
\end{equation}

Now, let $q^t := x^t + \rho(0, x^t_2, \ldots, x^t_{n-1})e_n \in \mathcal{F}$. Let $z_i := e_n \cdot \bar{e}_i$, then $e_n = \sum_{i=2}^n z_i \bar{e}_i$. Consequently, we have
\[
T^*(\rho(0, x^t_2, \ldots, x^t_{n-1})e_n) = \frac{1}{h}\rho(0, x^t_2, \ldots, x^t_{n-1})\left(a_1\left(\sum_{i=2}^n k_iz_i\right)e_1 + \sum_{i=2}^n a_iz_i \bar{e}_i\right).
\]
Since $|t| \leq C_\epsilon h^{\frac{1}{2} - 2\epsilon},$ by Lemma \ref{decayrate2} we have 
$\left|\rho(0, x^t_2, \ldots, x^t_{n-1})\right| \leq C_\epsilon h^{\frac{3}{4}-\frac{7}{2}\epsilon}.$
Given the above estimate for $\rho,$ $|t| \leq C_\epsilon h^{\frac{1}{2} - 2\epsilon}$, $|z_i| \leq 1$, and $|a_i| \leq C_\epsilon h^{\frac{1}{3} - \epsilon}$ for $i = 2, \ldots, n$, we deduce that
\[
\left|\frac{1}{h}\rho(0, x^t_2, \ldots, x^t_{n-1})\sum_{i=2}^n a_iz_i \bar{e}_i\right| \leq C_\epsilon h^{\frac{1}{12} - 10\epsilon}
\]
provided $h$ is sufficiently small.

Let \(\tilde{d}_e:=\sup\{|x \cdot e| : x \in S_h^c[v]\}\).
By applying the uniform density property, we have a comparison between the widths of \(S_h^c[v]\) and \(S_h^c[v]\cap\Omega_1^*\) in the \(e\) direction, asserting their comparability. Furthermore, Remark \ref{reeq1} indicates that the width of \(S_h[v]\cap \Omega_1^*\) in the \(e\) direction is also comparable to that of \(S_h^c[v]\cap\Omega_1^*\) in the same direction.
Synthesizing these comparative relationships with the estimates from Lemma \ref{ide1}, we are led to:
\begin{equation}\label{wid111}
\tilde{d}_e^2 \leq C a_1 \quad \forall \text{ unit vector } e \in \text{span}\{e_2, \ldots, e_n\}
\end{equation}
for a universal constant $C.$
Consequently, we obtain
\begin{equation}\label{estwid21}
|k_i| \leq C\frac{a_1^{1/2}}{a_1} = C a_1^{-1/2}.
\end{equation}
Note that \(|a_1| \leq C_\epsilon h^{2/3 - \epsilon}\). Therefore, we have
\begin{align*}
\left|\frac{1}{h} \rho(0, x_2^t, \ldots, x_{n-1}^t) a_1 \left(\sum_{i=2}^n k_i z_i\right) e_1\right| &\leq C_\epsilon \frac{1}{h} \left(C_\epsilon h^{1/2 - 2\epsilon}\right)^{3/2 - \epsilon} \left(C_\epsilon h^{2/3 - \epsilon}\right)^{1/2}\\
&\leq C_\epsilon h^{1/12 - 10\epsilon}.
\end{align*}
Consequently, 
\begin{equation}\label{flat1}
|T^*(q^t) - T^*(x^t)| \leq 2C_\epsilon h^{1/12 - 10\epsilon} \rightarrow 0 \quad \text{as} \quad h \rightarrow 0,
\end{equation}
provided that \(\epsilon\) is initially chosen to be sufficiently small.

By taking a subsequence, if necessary, we may assume that \(T^*(\text{span}\{e_2, \ldots, e_{n-1}\})\) converges to \(H\), an \((n-2)\)-dimensional subspace of \(\mathbb{R}^n\). Consequently, from \eqref{bp112}, \eqref{flat1}, and Corollary \ref{goodinclu}, it follows that \(H \subset \partial \Omega_{1,\infty}\). 
By convexity, we have that 
\begin{equation}\label{o1split}
\Omega_{1,\infty}=\omega\times H
\end{equation}
 for some two dimensional convex set $\omega.$

Now, we \emph{claim} that
\begin{equation}\label{gp11}
e_1 \notin H.
\end{equation}
Suppose to the contrary that the $x_1$-axis is contained in $H \subset \overline{\Omega_{1,\infty}}$. From the definition of $u_\infty$, it follows that $u_\infty(0) = 0$ and $u_\infty \geq 0$ on $H$. On one hand, by Remark \ref{uinfty1}, we have $u_\infty(-te_1) \leq u_\infty(0)=0$ for all small $t > 0$, and it follows that
\begin{equation}\label{crf1}
u_\infty(-te_1) = 0 \quad \text{for all sufficiently small } t > 0.
\end{equation}
On the other hand, Lemma \ref{uvreg1} ensures that $u_\infty$ is strictly convex in $B_r(0) \cap \overline{\Omega_{1,\infty}}$. This, however, contradicts \eqref{crf1}.

 Note that \(\frac{Te_n}{|Te_n|}\) is the unit inner normal to \(T^*\Omega_1\) at the origin. Without loss of generality, after passing to a subsequence of \(h\), we may assume that \(\frac{Te_n}{|Te_n|}\) converges to a unit vector \(e_\infty\), orthogonal to both \(H\) and \(e_1\). By Lemma \ref{saddle}, the line \(\{te_1 : t \in \mathbb{R}\}\) cannot intersect the interior of \(\Omega_{1, \infty}\). Thus,
\begin{equation}\label{gp22}
\Omega_{1, \infty} \subset \{x : x \cdot e_\infty \geq 0\}.
\end{equation}

Recall that the boundary of $\Omega^*_1$ is given by \(\partial \Omega^*_1 = \{x : x_1 = \rho^*(x_2, \ldots, x_n)\}\), where \(\rho^*\) is a \(C^2\), uniformly convex function in a neighbourhood of the origin, satisfying \(\rho^*(0) = 0\) and \(D\rho^*(0) = 0\). Denote \(x'' = (x_2, \ldots, x_n)\). In this setting, there exists a positive, homogeneous quadratic polynomial \(P'\) such that
\[\rho^*(x'') = P'(x'') + o(|P'(x'')|).\]
We can extend \(P'\) to \(\R^n\) by defining \(\tilde{P}(x_1, x'') = P'(x'')\). This leads to the following representation:
\[\partial \Omega^*_1 = \left\{x : \langle x, e_1\rangle= \tilde{P}(x) + \eta(x)\tilde{P}(x)\right\}\ \text{near the origin},\]
where $\eta$ is a function satisfying $\eta(x)\rightarrow 0$ as $x\rightarrow 0.$

A straightforward computation yields
\begin{equation}\label{h10}
\partial (T\Omega_1^*) = \left\{x : \langle x, e_1\rangle = \tilde{P}_h(x) + \eta(T^{-1}x)\tilde{P}_h\right\} \quad \text{near } 0,
\end{equation}
where \(\tilde{P}_h(x) = \frac{1}{|(T^t)^{-1}e_1|} \tilde{P}(T^{-1}x) \geq 0\), and
\begin{equation}\label{colla}
B_1(0) \cap T\Omega_1^* \subset \left\{x : \langle x, e_1\rangle \geq \frac{1}{2}\tilde{P}_h(x)\right\} \quad \text{for \(h > 0\) small}.
\end{equation}
By applying the uniform density property of \( v \) near the origin, and following a similar line of argument to the proof of \cite[Lemma 5.14]{CLW3}, we can deduce that the coefficients of \( \tilde{P}_h \) are bounded by a constant that is independent of \( h \).
 Consequently, upon passing to a subsequence if necessary, we may assume that \(\tilde{P}_h\) converges to a non-negative homogeneous quadratic polynomial \(P\), leading to
\begin{equation}\label{gp111}
\Omega_{1,\infty}^* = \{x \in \mathbb{R}^n : x_1 > P(x)\}.
\end{equation}

Since \( e_1 \cdot e_\infty = 0 \), by a rotation of coordinates we may assume \( e_\infty = e_n \).
By \eqref{gp11}, for a fixed unit vector \( e \in H \), we can find a vector \( \tilde{e} \in H^* := \text{span}\{e_2, \cdots, e_{n-2}\} \) such that \( e \) is not orthogonal to \( \tilde{e} \).
Hence, there exists an affine transformation \( A_1 \) with \( \det\, A_1 = 1 \) such that \( A_1(\text{span}\{e_n\}) = \text{span}\{e_n\} \), \( A_1(\text{span}\{e_1, \ldots, e_{n-1}\}) = \text{span}\{e_1, \ldots, e_{n-1}\} \),
and \( A_1e \) is parallel to \( (A_1^t)^{-1}\tilde{e} \) (see \cite[(3.2)]{CLW3}).
 The unit inner normals of \( A_1(\Omega_{1,\infty}) \) and \( (A_1^t)^{-1}\Omega^*_{1,\infty} \) at \( 0 \) are still orthogonal to each other. Denote \( \bar{e}_2 = \frac{A_1e}{|A_1e|} \). Then, \( A_1(\Omega_{1,\infty}) = \omega_1 \times H_1 \times \text{span}\{\bar{e}_2\} \), where \( \omega_1 \) is a two-dimensional convex subset and \( H_1 \) is an \( (n-3) \)-dimensional subspace in \( \mathbb{R}^n \).  It is easy to see that \( (A_1^t)^{-1}H^* = H^*_1 \times \text{span}\{\bar{e}_2\} \) for some \( (n-3) \)-dimensional subspace \( H^*_1 \)  in \( \mathbb{R}^n \).

Then we restrict ourselves to \( H_1 \) and \( H^*_1 \) in the \((n-2)\)-space \( \text{span}\{\bar{e}_2, e_n\}^\bot \).
Similarly, as above, we can find unit vectors \( e' \in H_1, \tilde{e}' \in H^*_1 \) and an affine transformation \( A_2 \) such that
\( A_2(\text{span}\{\bar{e}_2, e_n\}) = \text{span}\{\bar{e}_2, e_n\} \), \( A_2( \text{span}\{\bar{e}_2, e_n\}^\bot ) = \text{span}\{\bar{e}_2, e_n\}^\bot \),
and \( A_2e' \) is parallel to \( (A_2^t)^{-1}\tilde{e}' \).
Let \( \bar{e}_3 = \frac{A_2e'}{\|A_2e'\|} \). Repeating this process, after a sequence of affine transformations \( A_i \), \( i=1, \ldots, n-2 \), we have \( AH = (A^t)^{-1}H^* \), where \( A = A_{n-2} \cdots A_1 \).

Hence, by performing the affine transformation \( A \) as above, and then by a rotation of coordinates, we may assume that 
\( \Omega_{1, \infty}^* \) is as in the statement of Lemma \ref{gdspl1} and that
\( \Omega_{1,\infty} = \omega \times \text{span}\{e_2, \ldots, e_{n-2}\} \) for some two-dimensional convex set \( \omega \subset \text{span}\{e_1, e_n\} \) satisfying \( 0 \in \partial \omega \) and \( \omega \subset \{x_n \geq 0\} \).

Finally, after an affine transformation of the form
\begin{equation*}
\tilde{A}: \left\{ \begin{array}{rll}
x_1 &\to x_1 + k x_n & \text{for a constant \( k \in \mathbb{R} \)} \\
x_i &\to x_i & \text{for } i = 2, \ldots, n.
\end{array}
\right.
\end{equation*}
we can transform \(\Omega_{1,\infty}\) to the position as in the statement of the lemma.
Note that \(\tilde{A}\) makes \(\Omega_{1,\infty}\) slide along the \(x_1\)-direction, and at the same time \((\tilde{A}^t)^{-1}\) makes \(\Omega^*_{1,\infty}\) slide along the \(y_n\)-direction, while the \((n-2)\)-space \(\text{span}\{e_2, \ldots, e_{n-2}\}\) remains invariant.
Hence, by choosing a proper constant \( k \in \mathbb{R} \), we may assume that \(\omega = \{ (x_1, x_n) : x_n \geq \rho_\infty(x_1) \}\) for a convex function \(\rho_\infty\). Note that since \( P \) is a non-negative homogeneous quadratic polynomial, after the corresponding affine transformation \((\tilde{A}^t)^{-1}\), the set \(\Omega^*_{1,\infty}\) still satisfies \(\Omega_{1,\infty}^* = \{ x \in \mathbb{R}^n : x_1 > P(x) \}\), but with a different non-negative homogeneous quadratic polynomial \( P \).
\end{proof}

In the sequel, for notational brevity, we shall denote \( v_\infty \)  by \( \tilde{v} \). Define the sets 
	\[ V := B_1(0)\cap\Omega^*_{1,\infty}, \quad\text{ and }\  U := D\tilde{v} (V) \subset \Omega_{1,\infty}. \]
Note that $V$ is a bounded convex set.
By the strict convexity of $\tilde v$, there exists a small \( r_1 \in (0, r_0) \) such that \( B_{r_1}(0) \cap \Omega_{1,\infty} \subset U \). It follows that $\partial U$ is convex near $0.$

By Lemma~\ref{gdspl1}, in a neighbourhood of \(0\), we have 
\begin{equation}\label{partialU}
U = \left\{x \in \mathbb{R}^n : x_n > \rho_\infty(x_1)\right\},
\end{equation}
where \(\rho_\infty\) is a convex function satisfying \(\rho_\infty(0) = 0\) and \(\rho_\infty \geq 0\), and
\begin{equation}\label{partialV}
V = \left\{y \in \mathbb{R}^n : y_1 > \bar{\rho}(y_2, \ldots, y_n)\right\},
\end{equation}
where \(\bar{\rho}\) is a smooth convex function satisfying \(\bar{\rho}(0) = 0\), \(D\bar{\rho}(0) = 0\).

Let \(\tilde u:\mathbb{R}^n \to \mathbb{R}\) be defined by
\begin{equation}\label{tildeu}
\tilde u(x) := \sup \{x \cdot y - \tilde v(y) : y \in V\} \quad \forall\, x \in \mathbb{R}^n.
\end{equation}
By the definition of \(u_\infty\) and Lemma \ref{uvreg1}, we have \(\tilde u = u_\infty\) on \(B_{r_1}(0) \cap \overline{\Omega_{1,\infty}} \subset \overline{U}\).

Note also that
\begin{equation}\label{mau}
\begin{cases}
\det D^2\, \tilde{u} = c_1\chi_{U} \quad \text{in } \mathbb{R}^n, \\
D\tilde{u}(\mathbb{R}^n) = \overline{V}
\end{cases}
\end{equation}
for some positive constant \( c_1 \).
Denote \(r_2 = \frac{1}{2}r_1\).
Since \(\tilde u\) is strictly convex in \(B_{r_1}(0) \cap \overline{U} = B_{r_1}(0) \cap \overline{\Omega_{1,\infty}}\), there exists a constant \(h_0 > 0\) such that
\begin{equation}\label{loctidu1}
S_h^c[\tilde u](x) \cap U \subset B_{r_1}(0) \cap U \quad \forall\, x \in B_{r_2}(0) \cap \overline{U} \text{ and } \forall\, h \leq h_0.
\end{equation}
It follows from  \eqref{mau} and \eqref{loctidu1} that the Monge-Amp\`ere measure $\det\, D^2 \tilde{u}$ is doubling for 
centred sections $S_h^c[\tilde u](x)$, where $h\leq h_0$ and $ x\in B_{r_2}(0) \cap \overline{U}.$

We proceed to summarise the properties of \( \tilde{u} \) and \( \tilde{v} \):
\begin{enumerate}
    \item \textbf{Geometric decay.} By the obtained doubling property of the Monge-Amp\`ere measures $\det\, D^2\tilde u$ and $\det\, D^2\tilde v$, from \cite[Lemma 2.2]{C96} the centred sections \( S_h^c[\tilde{u}] \) and \( S_h^c[\tilde{v}] \) decay geometrically. Consequently, \( \tilde{u} \) (resp., \( \tilde{v} \)) is \( C^{1,\beta} \)-regular for some $\beta\in(0,1)$ and strictly convex in \( B_{r}(0) \cap \overline{U} \) (resp., \( B_{r}(0) \cap \overline{V} \)) for some positive \( r > 0 \).  
    
Applying \cite[Corollary 2.2]{C96}, we have the duality between $S_h^c[\tilde{v}]$ and $D\tilde{v}(S_h^c[\tilde{v}])$: if 
	\[ E\subset  S_h^c[\tilde v] = \{\tilde{v} < \ell\} \subset C_nE\quad\text{ for some affine function $\ell$}, \]
where $E$ is an ellipsoid centred at $0$ with principal radii $\lambda_i\tilde e_i$, $i=1,\cdots, n$, then
	\[ E^*\subset D\tilde{v}(S_h^c[\tilde v])\subset C_n E^*, \]
where $E^*$ is an ellipsoid centred at $D\ell$ with principal radii $\frac{h}{\lambda_i}\tilde e_i$, $i=1,\ldots, n$.
A similar duality holds for $S_h^c[\tilde{u}]$ and $D\tilde{u}(S_h^c[\tilde{u}])$ as well.
 
By the proof of \cite[Lemma 2.2]{CLW1}, we have
\begin{align}
S^c_{b^{-1}h}[\tilde v] \cap V \subset S_h[\tilde v] \cap V \subset S^c_{bh}[\tilde v], \label{equi02} \\
S^c_{b^{-1}h}[\tilde u] \cap U \subset S_h[\tilde u] \cap U \subset S^c_{bh}[\tilde u] \label{equi03}
\end{align}
for some constant \( b > 0 \) independent of \( h \).

    \item \textbf{Uniform density.} The affine invariance of the uniform density property ensures that
    \[ \frac{|S_h^c[\tilde{v}] \cap V|}{|S_h^c[\tilde{v}]|} \geq \delta \]
    for a constant \( \delta > 0 \) independent of \( h \). Indeed, it follows by taking limit of \eqref{udvh}.
    Similarly to \eqref{vol}, we have \begin{equation}\label{vol1}
|S_h[\tilde v]\cap V| \approx |S_h^c[\tilde v]\cap V| \approx |S_h^c[\tilde v]| \approx h^{\frac{n}{2}}.
\end{equation}

    To show the uniform density for \(\tilde u\), we proceed as follows. Let \(M>0\) be a large constant to be determined. In the following, we always assume $h$ is small such that \eqref{loctidu1} holds. 
     By performing an affine transformation, we may assume that
	\[S_{\frac{h}{M}}^c[\tilde{v}] \approx B_{\left(\frac{h}{M}\right)^{\frac{1}{2}}}(0).\]
Using the duality between \(S_{\frac{h}{M}}^c[\tilde{v}]\) and \(D\tilde{v}(S_{\frac{h}{M}}^c[\tilde{v}]),\) we find that
\begin{equation}\label{levcon2}
0 \in D\tilde{v}\left(S_{\frac{h}{M}}^c[\tilde{v}]\right) \approx B_{\left(\frac{h}{M}\right)^{\frac{1}{2}}}(z)
\end{equation}
for some \(z \in U.\) It follows that 
$\left|D\tilde{v}\left(S_{\frac{h}{M}}^c[\tilde{v}]\right)\right| \approx (h/M)^{\frac{n}{2}}.$

For any \(y \in S_{\frac{h}{M}}^c[\tilde{v}],\) from \eqref{tildeu} we have
\[\tilde{u}(D\tilde v(y)) = y \cdot D\tilde v(y) - \tilde v(y) \leq CM^{-1}h < b^{-1}h\]
by choosing \(M>Cb\) to be a universal constant independent of $h$. Hence,
\begin{equation}\label{levcon1}
D\tilde{v}\left(S_{\frac{h}{M}}^c[\tilde{v}] \right) \subset S_{b^{-1}h}[\tilde u] \cap \overline{U} \subset S^c_h[\tilde u]\cap \overline U.
\end{equation}
Combining the above estimates, we conclude that
\[ |S^c_h[\tilde u]\cap U| \geq C_1h^{\frac{n}{2}} \]
for some constant \(C_1\) independent of \(h.\)

By the Alexandrov type estimate (see, for instance, \cite[estimate (4)]{FK1}), we obtain
\begin{align*}
h^n &\geq C\left|\left(\frac{1}{2}S^c_h[\tilde u]\right) \cap U\right| \cdot \left|S^c_h[\tilde u]\right| \\
&\geq C_2\left|S^c_h[\tilde u] \cap U\right| \cdot \left|S^c_h[\tilde u]\right|,
\end{align*}
where the second inequality follows from the convexity of $U$ near $0$. Combining the above two estimates, we obtain the uniform density estimate for $\tilde u$:
\[
\frac{\left|S_h^c[\tilde u] \cap U\right|}{\left|S_h^c[\tilde u]\right|} \geq \delta
\]
for some $\delta > 0$, independent of $h$.

    \item \textbf{Conjugation between \( S^c_h[\tilde{u}] \) and \( S^c_h[\tilde{v}] \).}
    The uniform density property implies that \( S^c_h[\tilde{u}] \) and \( S^c_h[\tilde{v}] \) are conjugate for sufficiently small \( h \). More precisely, if there exists an affine transformation \( A \) with \( \det\, A = 1 \) such that
\( AS^c_h[\tilde{v}] \approx B_{h^{\frac{1}{2}}}(0) \), then it follows that \( (A^t)^{-1}S^c_h[\tilde{u}] \approx B_{h^{\frac{1}{2}}}(0) \). 

To see why this is the case, let \( A \) be the affine transformation such that \( \det\, A = 1 \) and \( AS^c_h[\tilde{v}] \approx B_{h^{\frac{1}{2}}}(0) \), and define
	\[ \bar{v}(y) := \tilde{v}(A^{-1}y),\quad\text{ and }\ \ \bar{u}(x) := \tilde{u}(A^t x). \] 
Then \( D\bar{v} \) solves the optimal transport problem between the distributions of \( AV \) and \( (A^t)^{-1}U \).
Note that now \( S^c_h[\bar{v}] \approx B_{h^{\frac{1}{2}}}(0) \). We aim to show \( S^c_h[\bar{u}] \approx B_{h^{\frac{1}{2}}}(0) \) as well.

Let \( M \) be as in \eqref{levcon1}. By observation (b) in the proof of \cite[Lemma 4.1]{C96}, we have  
\[
\frac{M}{C_n}S_{\frac{h}{M}}^c[\bar v] \subset S_h^c[\bar v] \subset C_nMS_{\frac{h}{M}}^c[\bar v].
\]
Hence, for some constant \( C_1 \) independent of \( h \), we have
\[
B_{C_1^{-1}h^{\frac{1}{2}}}(0) \subset S_{\frac{h}{M}}^c[\bar v] \subset B_{C_1h^{\frac{1}{2}}}(0).
\]
By the same argument leading to \eqref{levcon2} and \eqref{levcon1}, we have that for some \( z \in U \) and some constant \( c_1 > 0 \),
\[
B_{c_1h^{\frac{1}{2}}}(z) \subset D\bar v\left(S_{\frac{h}{M}}^c[\bar v] \right) \subset S^c_h[\bar u].
\]
By the uniform density property, we have \( |S^c_h[\bar u]| \approx h^{\frac{n}{2}} \).
From the fact that \( S^c_h[\bar u] \) is balanced with respect to the origin, it follows that
\[
(A^t)^{-1}S^c_h[\tilde{u}] = S^c_h[\bar u] \approx B_{h^{\frac{1}{2}}}(0).
\]

Additionally, we can conclude that: for any \( x \in S^c_h[\tilde{u}] \) and \( y \in S^c_h[\tilde{v}] \), as \( |(A^t)^{-1} x| \approx h^{\frac{1}{2}} \) and \( |A y| \approx h^{\frac{1}{2}} \), it follows that 
\begin{equation}\label{dualex1}
|x \cdot y|=|((A^T)^{-1} x)\cdot (A y)| \leq Ch \quad \forall\, x \in S^c_h[\tilde{u}],\ \forall\, y \in S^c_h[\tilde{v}]
\end{equation}
for some constant \( C \) independent of \( h \).     
    
    \item \textbf{Tangential $C^{1, 1-\epsilon}$ estimate.} 
    Owing to the flatness of \( \partial U \) near the origin in the directions of \( e_2, \ldots, e_{n-1} \), we invoke \cite[Corollary 1.1]{C96} to deduce that \( \tilde{u} \) is \( C^{1,1} \) in these directions. Consequently, for any small \( h>0 \), we have
\[
C_1h^{\frac{1}{2}}e \in S_h^c[\tilde{u}]\quad  \forall\, e \in \text{span}\{e_2,\ldots,e_{n-1}\}
\]
where \( C_1>0 \) is a constant independent of \( h \).
Then, by \eqref{dualex1},  
\[
C_1h^{\frac{1}{2}}e \cdot y \leq Ch \quad \forall\, y \in S_h^c[\tilde{v}].
\]
Therefore,
\begin{equation}\label{kuanest1}
\sup\{|y \cdot e| : y \in S_h^c[\tilde{v}]\} \leq Ch^{\frac{1}{2}} \quad  \forall\, e \in \text{span}\{e_2,\ldots,e_{n-1}\}. 
\end{equation}
Since \( \partial V \) is smooth near the origin, we apply the result in \cite[Lemma 3.1]{CLW1} to infer that \( \tilde{v} \) is \( C^{1,1-\epsilon} \) tangentially for any $\epsilon>0$. More precisely, 
\begin{equation}\label{kuanest2}
B_{C_\epsilon h^{\frac{1}{2}+\epsilon}}(0) \cap \{x_1=0\} \subset S_h^c[\tilde{v}].
\end{equation}
  \end{enumerate}

\vskip 20pt
{\bf Second blow-up.} 
Suppose that \(S_h^c[\tilde v]\) is comparable to an ellipsoid \(E\) centred at the origin, which is given by \eqref{elliptan}. 
Thanks to estimates \eqref{kuanest1} and \eqref{kuanest2}, we can improve the bounds in \eqref{ai000} such that 
\begin{align}
0 < & a_1 \leq C_\epsilon h^{\frac{2}{3}-\epsilon}, \nonumber \\
c_\epsilon h^{\frac{1}{2}+\epsilon} \leq & a_i \leq C_\epsilon h^{\frac{1}{2}-\epsilon} \quad \text{ for } i = 2, \ldots, n-1, \label{ai1} \\ 
c_\epsilon h^{\frac{1}{2}+\epsilon} \leq & a_n \leq C_\epsilon h^{\frac13-\epsilon}. \nonumber
\end{align}
Furthermore, the \(C^{1,\beta}\) regularity of \(\tilde{v}\) within \(B_{r_0}(0) \cap V\) ensures that
\begin{equation}\label{ai2}
B_{Ch^{\frac{1}{1+\beta}}}(0) \cap V \subset S_h^c[\tilde{v}].
\end{equation}

For the given $h>0$ small, let $A_{h1}=T_1$ defined in \eqref{afT1}, $A_{h2}=T_2$ defined in \eqref{afT2}, and $A_h = A_{h2} \circ A_{h1}$
that maps \(E\) onto the unit ball \(B_1(0)\). 
Define \(\tilde v_h: \mathbb{R}^n \rightarrow \mathbb{R}\) by 
\[\tilde v_h(x) := \frac{1}{h} \tilde v(A_h^{-1}x).\] 

Take any point \(x = (x_1, x'') \in \{x_n = 0\} \cap \partial V\) such that \(|x''| \leq h^{\frac{1}{2}-5\epsilon}\).
From \eqref{partialV},
\[x_1 = \bar\rho(x'') \leq Ch^{1-10\epsilon}.\] 
By \eqref{ai2}, we deduce that \(\|A_h\| \leq \frac{1}{C} h^{-\frac{1}{1+\beta}}\), then it follows that
\begin{equation}\label{fat1}
\text{dist}\left(A_hx, A_h\text{span}\{e_2, \ldots, e_{n-1}\}\right) \leq C_\epsilon h^{1-10\epsilon-\frac{1}{1+\beta}} \to 0\quad \text{ as } h \to 0,
\end{equation}
provided \(\epsilon\) is sufficiently small.
Meanwhile, for \(x = (x_1, x'') \in \{x_n = 0\} \cap \partial V\) with \(|x''| = h^{\frac{1}{2}-5\epsilon}\), we have 
\begin{equation}\label{fat2}
|A_hx| \geq C_\epsilon h^{\frac{1}{2}-5\epsilon-(\frac{1}{2}-\epsilon)} =  C_\epsilon h^{-4\epsilon} \to \infty \quad \text{ as } h \to 0,
\end{equation}
provided \(\epsilon > 0\).

Thanks to \eqref{fat1} and \eqref{fat2}, and by convexity of $V$, we can conclude that the sequence (possibly passing to a subsequence) \( A_h(\{x_n=0\} \cap \partial V)\} \) converges locally uniformly to an \((n-2)\)-dimensional subspace \(\tilde{H}^*\) of \(\mathbb{R}^n\). Invoking convexity once more and further passing to a subsequence if necessary, we may assume that \( A_h V \) converges locally uniformly to a smooth convex set represented by \(\tilde{\Omega}^* = \omega^* \times \tilde{H}^*\), where \(\omega^*\) is a two-dimensional convex set with smooth boundaries. The smoothness of \(\omega^*\) can be inferred through an argument analogous to that of \eqref{gp22} to \eqref{gp111}.


Invoking convexity once more and passing to a subsequence, we can assume that the set \((A_h^t)^{-1}(U \cap B_{r_1}(0))\) converges locally uniformly to a convex set \(\tilde{\Omega}\) in the Hausdorff distance. According to \eqref{partialU}, we find that
\[
\{x \in \mathbb{R}^n : x_1 = x_n = 0\} \cap B_{r_1}(0) \subseteq U \cap B_{r_1}(0),
\]
and by passing to a further subsequence if necessary, we may deduce that $$(A_h^t)^{-1}\left(\{x \in \mathbb{R}^n : x_1 = x_n = 0\}\right)$$ converges locally uniformly to an \((n-2)\)-dimensional subspace \(\tilde{H}\) of \(\mathbb{R}^n\) in the Hausdorff distance.
Consequently, \(\tilde{H} \subset \partial \tilde{\Omega}\), and by convexity, it follows that
\(\tilde{\Omega} = \omega \times \tilde{H}\) for some two-dimensional convex set \(\omega\).
Furthermore, upon passing to another subsequence if necessary, we may assume that \((A_h^t)^{-1}e_1\) and \(A_he_n\) converge to the vectors \(\tilde{e}^*\) and \(\tilde{e}\), respectively, such that \(\tilde{e}^* \perp \tilde{e}\) and
\begin{align*}
\tilde{\Omega} &\subseteq \{x\in\mathbb{R}^n : x \cdot \tilde{e} > 0\}, \\
\tilde{\Omega}^* &\subseteq \{x\in\mathbb{R}^n : x \cdot \tilde{e}^* > 0\}.
\end{align*}

Similarly to the discussion after \eqref{gp111} in the proof of Lemma \ref{gdspl1} (or see the proof of \cite[Lemma 5.19]{CLW3}),  by performing an additional affine transformation, we further have 
\begin{equation}
\tilde{\Omega} = \{x \in \mathbb{R}^n : x_n \geq \tilde{\rho}(x_1)\}, 
\end{equation}
where \(\tilde\rho \geq 0\) is a convex function defined in a neighbourhood of \(0\) with \(\tilde\rho(0) = 0\), and
\begin{equation}
\tilde{\Omega}^* = \{x \in \mathbb{R}^n : x_1 \geq \tilde{\rho}^*(x_n)\},
\end{equation}
where \(\tilde\rho^* \geq 0\) is a smooth convex function defined in a neighbourhood of \(0\) with \(\tilde\rho^*(0) = 0\).

Additionally, we may assume that \(\tilde v_h\) converges to \(\tilde v_\infty\) locally uniformly, where \(\tilde v_\infty\) satisfies \(\tilde v_\infty(0) = 0\), \(\tilde v_\infty \geq 0\), and
\begin{equation}\label{maulim}
\begin{aligned}
\det\, D^2 \tilde v_\infty &= c_0 \chi_{_{\tilde{\Omega}^*}} \quad \text{in } \mathbb{R}^n, \\
D\tilde v_\infty(\tilde{\Omega}^*) &= \tilde{\Omega}
\end{aligned}
\end{equation}
for some constant \(c_0 > 0\).

The above blow-up limits satisfy all the conditions required in \cite[Proposition 5.1 (ii)]{CLW3}.
Finally, a contradiction can be derived by following the argument presented in \cite[Section 6.2]{CLW3}, concluding that Case \textbf{I} cannot occur.

\subsection{Case \textbf{II}: $\nu\cdot \nu^*=0,\ \hat\nu \cdot \hat\nu^* = 0$, and $\nu, \nu^*, \hat\nu^*$ are coplanar.}

As before, we assume:
\(0 \in \mathcal{F}\), \(y_0 = Du_1(0) = le_n \in \partial\Omega^*_1\), and \(\hat{y}_0 = Du_2(0) = -le_n \in \partial\Omega^*_2\), where \(l>0\) is a universal constant;
the unit normals \(\nu = \nu(0) = e_n\), \(\nu^* = \nu^*(y_0) = e_1\) and \(\hat{\nu}^* = \hat{\nu}^*(\hat{y}_0) = e_1\) or \(-e_1\). By subtracting a constant, we may assume $u_i(0)=0$, $i=1, 2$ and $v(y_0)=v(\hat y_0)=0.$
 
Define the quantities
\begin{align*}
d_h &:= \sup\{|y\cdot e_1| : y\in S_h[v]\cap \Omega_1^*\}, \\
\hat{d}_h &:= \sup\{|y\cdot e_1| : y\in S_h[v]\cap \Omega_2^*\}.
\end{align*}
We then have two possible scenarios: $d_h\geq \hat{d}_h$ or $d_h<\hat{d}_h$.

Consider a sequence $\{h_k\}\rightarrow 0$ such that either $d_k:=d_{h_k}\geq \hat{d}_{h_k}=:\hat d_k$ holds for all $k$, or the opposite inequality holds for all $k$. Assume without loss of generality that the former is true for all $k$. 
For any unit vector $e\in \text{span}\{e_2, \cdots, e_n\}$, define
\begin{align*}
d_{k,e} &:= \sup\{|(y-y_0)\cdot e| : y\in S_{h_k}[v]\cap \Omega_1^*\}, \\
\hat{d}_{k,e} &:= \sup\{|(y-\hat y_0)\cdot e| : y\in S_{h_k}[v]\cap \Omega_2^*\}.
\end{align*}
By the proof of Lemma \ref{ide1}, we can obtain 
\begin{equation}\label{estwid1}
d_{k, e},\ \hat{d}_{k, e} \leq C_{\epsilon}h^{\frac{1}{3}-\epsilon} \quad \forall\, e\in \text{span}\{e_2, \cdots, e_n\}.
\end{equation}

Suppose the centred section $S_{h_k}^c[v](y_0)$ is comparable to $E+\{y_0\}$, where $E$ is an ellipsoid centred at the origin. 
Similarly as before, the ellipsoid $E$ can be expressed by \eqref{elliptan}, namely
\begin{equation*} 
E = \left\{y = y_1e_1 + \sum_{i=2}^n \bar{y}_i \bar{e}_i : \frac{y_1^2}{a_1^2} + \sum_{i=2}^n \frac{(\bar{y}_i - k_i y_1)^2}{a_i^2} \leq 1\right\},
\end{equation*}
where $\bar e_2, \cdots, \bar e_n$ are the principal directions of  $E\cap\{x_1=0\}.$
Using Lemma \ref{ide1} and the tangential $C^{1,1-\epsilon}$ estimate of $v$ at $y_0$, we have the same estimates as in \eqref{ai000} that
\begin{equation}\label{estwid333}
0 < a_1 < C_\epsilon h_k^{\frac{2}{3}-\epsilon}, \quad C_\epsilon h_k^{\frac{1}{2}+\epsilon} < a_i < C_\epsilon h^{\frac{1}{3}-\epsilon}\ \  \text{ for } i=2, \cdots, n.
\end{equation}
By a similar reasoning as in \eqref{estwid21}, it follows that
\begin{equation}\label{estwid3}
|k_i| \leq C a_1^{-\frac{1}{2}}, \quad i = 2, \cdots, n.
\end{equation}
Finally, we remark that from the above assumption and definition,
\begin{equation}\label{estwid4}
\hat d_k \leq d_k \approx a_1. 
\end{equation}

For each $k=1,2,\cdots$, let $T_{k1}$ be the transformation given in \eqref{afT1} and $T_{k2}$ be the transformation given in \eqref{afT2}, and $T_k = T_{k2} \circ T_{k1}$.
Then \( T_k(E) = B_1(0) \),  and \( T_k\left(S_{h_k}^c[v](y_0)\right) \approx B_1(T_k y_0) \).
Let  \( \Omega^*_{ki} := T_k(\Omega^*_i) \), for \( i = 1,2 \),
and \( z_k := T_k y_0 \in \partial \Omega^*_{k1} \).
By Corollary \ref{co21}, it follows that 
\begin{equation}\label{zklev1}
 B_{\frac{1}{C}}(z_k) \cap \Omega^*_{k1} \subseteq T_k(S_{h_k}[v] \cap \Omega_1^*) \subseteq B_C(z_k).
 \end{equation}  

To analyse the scaling behaviour, we define the rescaled functions
\[
v_k(\cdot) = \frac{1}{h_k} v\left(T_k^{-1}(\cdot)\right); \quad u_{ki}(\cdot) = \frac{1}{h_k} u_i\left(T_k^t(\cdot)\right), \quad i=1, 2.
\]
Denote \(\Omega_{ki} = \frac{1}{h_k}(T_k^t)^{-1}\Omega_i\), for \(i = 1, 2\), and \(\mathcal{F}_k = \frac{1}{h_k}(T_k^t)^{-1}\mathcal{F}\). 
For any point \(x \in \mathcal{F}_k\), the unit normal \(\nu_k\) to \(\mathcal{F}_k\) satisfies 
	\[ \nu_k(x) = \frac{Du_{k1}(x) - Du_{k2}(x)}{|Du_{k1}(x) - Du_{k2}(x)|}. \]  
After a suitable rotation of coordinates that aligns \(T_k e_n\) with the \(e_n\)-axis, we have 
	\begin{equation}\label{go2i}
		z_k = l_k e_n,\quad\text{ where }\ l_k \to \infty\ \text{ as } k \to \infty. 
	\end{equation}
Moreover, at the points \(z_k=Du_{k1}(0)\) and \(-z_k=Du_{k2}(0)\), the unit inner normals to \(\partial \Omega^*_{ki}\) for \(k = 1, 2\) are parallel to \(e_1\). 
By Lemma \ref{c1lemma}, we have that \( v_k \) is continuously differentiable in \( \left(B_1(z_k)\cap \overline{\Omega^*_{k1}}\right)\cup\left(B_1(-z_k)\cap \overline{\Omega^*_{k2}}\right) \) for \( k \) large.

Since \(d_k \geq \hat{d}_k  \) and $S_1[v_k]\cap\Omega^*_{k1}$ is normalised in the sense of \eqref{zklev1}, we conclude
	\begin{equation} \label{bde1}
		\sup\{ |y\cdot e_1| : y\in S_1[v_k]\cap\Omega^*_{ki}, \ i=1,2 \} \leq C \quad\forall\,k=1,2,\cdots, 
	\end{equation}
where $C>0$ is a universal constant. 
\begin{lemma}\label{limitdist}
The distance between \(S_1[v_k] \cap \Omega^*_{k1}\) and \(S_1[v_k] \cap \Omega^*_{k2}\) satisfies:
  \[
  \frac{\mathrm{dist}\big(S_1[v_k] \cap \Omega^*_{k1},\ S_1[v_k] \cap \Omega^*_{k2}\big)}{l_k} \rightarrow 2 \quad \text{as } k \rightarrow \infty.
  \]
  Moreover, \(\text{diam}\left(S_1[v_k] \cap \Omega^*_{k1}\right) \ll l_k\) and \(\text{diam}\left(S_1[v_k] \cap \Omega^*_{k2}\right) \ll l_k\) for \(k\) large.
\end{lemma}

\begin{proof}
  By the definition of \(T_{k1}\) and estimates \eqref{estwid333}--\eqref{estwid3}, the width of \(T_{k1}(S_{h_k} \cap \Omega_2^*)\) in the \(e_i\)-direction is upper bounded by \(C_\epsilon h_k^{\frac{1}{3}-\epsilon}\) for \(i = 2, \dots, n\). 
  Subsequently, by \eqref{estwid333} again and the definitions of \(T_{k2}\), we infer that the width of \(S_1[v_k] \cap \Omega^*_{k2} = T_k(S_{h_k} \cap \Omega_2^*)\) in the \(e_i\)-direction is bounded from above by:
  \begin{equation}\label{estwid5}
    \frac{C_\epsilon h_k^{\frac{1}{3}-\epsilon}}{\min\{a_i : i=2, \dots, n\}} \leq C_\epsilon h_k^{-\frac{1}{6}-2\epsilon} \quad \text{for } i = 2, \dots, n.
  \end{equation}
  Additionally, from \eqref{bde1}, the width of \(S_1[v_k] \cap \Omega^*_{k2}\) in the \(e_1\) direction is upper bounded by a constant \(C\) independent of \(k\). 
  
By the definition of \(T_{k}\) and \eqref{estwid333}, one can see that \(l_k \geq C_\epsilon h_k^{-\frac{1}{3}-\epsilon} \gg C_\epsilon h_k^{-\frac{1}{6}-2\epsilon}\).
 Hence, \(\text{diam}\left(S_1[v_k] \cap \Omega^*_{k1}\right) \ll l_k\) and \(\text{diam}\left(S_1[v_k] \cap \Omega^*_{k2}\right) \ll l_k\) for \(k\) large.
Therefore, from the above discussions we can conclude that:
  \[
  \frac{\mathrm{dist}(S_1[v_k] \cap \Omega^*_{k1}, S_1[v_k] \cap \Omega^*_{k2})}{l_k} = \frac{|z_k-(-z_k)|}{l_k} +o(1) \rightarrow 2 \quad \text{as } k \rightarrow \infty.
  \]
\end{proof}

Similar to \eqref{ccin1}, there exists an open cone $\mathcal{C}_k$ with vertex $z_k$ and the size of opening independent of $k,$ such that 
\begin{equation}\label{conek}
B_1(z_k)\cap \mathcal{C}_k\subset \Omega_{k1}^*.
\end{equation}

\begin{lemma}\label{vest1}
  There exist constants \(\beta \in (0, 1)\), \( a>1 \), \(r_1>0\) and \(C > 0\) (independent of \(k\)) such that for any \( y \in B_{r_1}(z_k)\), we have
  \begin{equation}\label{alpha2}
    0 \leq v_k(y) \leq C |y - z_k|^{1 + \beta},
    \end{equation}
 and for any \(y \in B_{r_1}(z_k) \cap \overline{\Omega_{k1}^*}\),
  \begin{equation}\label{beta1}
   v_k(y) \geq \frac{1}{C} |y - z_k|^{1 + a}.
  \end{equation}
\end{lemma}
\begin{proof}
By applying an argument similar to the proof of Lemma \ref{vesttoday}, we can obtain 
	\begin{equation}\label{vkes}
		\frac{1}{C} |y - z_k|^{1 + a} \leq v_k(y) \leq C |y - z_k|^{1 + \beta} \quad \forall\,y \in B_{r_1}(z_k) \cap \overline{\Omega_{k1}^*}
	\end{equation}
with constants \(\beta \in (0, 1)\), \(a > 1\), \(r_1 > 0\), and \(C > 0\) independent of \(k\). 

Given that relation \eqref{equi0} remains invariant under affine transformations, the function \( v_k \) preserves this relation with the same constant \( b \) in \eqref{equi0}. 
By \eqref{vkes}, we deduce that
\begin{equation}\label{ccin2}
B_{c_1\tilde{h}^{\frac{1}{1+\beta}}} (z_k)\cap \Omega^*_{k1} \subset S_{\frac{1}{b}\tilde{h}}[v_k](z_k)\cap\Omega^*_{k1} \subset S^c_{\tilde{h}}[v_k](z_k) \cap \Omega^*_{k1}\quad\forall\,\tilde{h}\in (0, \tilde h_0),
\end{equation}
where $\tilde h_0$, $c_1 = (Cb)^{-\frac{1}{1+\beta}}$ are constants independent of \( k \).

Combining inequalities \eqref{conek} and \eqref{ccin2}, we infer that
\[ 
B_{c_1\tilde{h}^{\frac{1}{1+\beta}}} (z_k)\cap \mathcal{C}_k\subset S^c_{\tilde{h}}[v_k](z_k)\quad\forall\,\tilde{h}\in (0, \tilde h_0).
\]
Taking into account that \( S^c_{\tilde{h}}[v_k](z_k) \) is centred at \( z_k \), we conclude that the ``opposite" cone
\[
W_k := -\frac{1}{C_n}\left( \big(B_{c_1\tilde{h}^{\frac{1}{1+\beta}}} (z_k) \cap \mathcal{C}_k\big) - \{z_k\} \right) + \{z_k\} \subset S^c_{\tilde{h}}[v_k](z_k)
\]
for some large constant \( C_n \), dependent only on \( n \).
Thanks to the convexity of \( S^c_{\tilde{h}}[v_k](z_k) \), it must contain the convex hull of
\[ 
\big(B_{c_1\tilde{h}^{\frac{1}{1+\beta}}} (z_k)\cap \mathcal{C}_k\big) \cup W_k.
\]
Since the size of opening of the cone $\mathcal{C}_k$ is independent of $k$, one can see that
\[ 
B_{\frac{1}{C}\tilde{h}^{\frac{1}{1+\beta}}}(z_k) \subset \text{convex hull of}\  \big(B_{c_1\tilde{h}^{\frac{1}{1+\beta}}} (z_k)\cap \mathcal{C}_k\big) \cup W_k,
\]
where \( C \) is a constant independent of \( k \). Consequently, we obtain
	\begin{equation} \label{esvout}
		B_{\frac{1}{C}\tilde{h}^{\frac{1}{1+\beta}}}(z_k) \subset S^c_{\tilde h}[v_k]\subset S_{b\tilde{h}}[v_k]
	\end{equation}
for $k$ large. The desired estimate \eqref{alpha2} readily follows from \eqref{esvout}.
\end{proof}

\begin{lemma}\label{alpha1}
  There exists a universal constant \(N\) such that for \(u_{k1}-l_kx_n\) and \(u_{k2}+l_kx_n\), we have the inclusions
  \begin{align}
    S_{\frac{1}{N}}[u_{k1}-l_kx_n](0) &\cap \mathcal{F}_k \subset Dv_k(S_1[v_k] \cap \overline{\Omega^*_{k1}}), \label{vust} \\
    S_{\frac{1}{N}}[u_{k2}+l_kx_n](0) &\cap \mathcal{F}_k \subset Dv_k(S_1[v_k] \cap \overline{\Omega^*_{k2}}), \label{vust2}
  \end{align}
  respectively.
\end{lemma}
\begin{proof}
By the duality between \(u_{k1}\) and \(v_k\) (see \eqref{dual222}), we have
\[ u_{k1}(x) = \sup\{x\cdot y - v_{k}(y) : y \in \Omega^*_{k1}\}. \]
By the strict convexity estimate of \( v_k \) (see \eqref{beta1}), we have 
\begin{equation}\label{vust8}
B_{r_2}(0) \cap \overline{\Omega_{k1}} \subset Dv_k\left(S_1[v_k] \cap \overline{\Omega^*_{k1}}\right)
\end{equation}
for some constant \( r_2 > 0 \) independent of \( k \).  
Hence, for any \( x \in B_{r_2}(0) \cap \overline{\Omega_{k1}} \),
\begin{align*}
u_{k1}(x) &= \sup\{x\cdot y - v_k(y) : y \in B_{r_1}(z_k) \cap \Omega_{k1}^*\} \\
&= \sup\{x\cdot y - v_k(y) : y \in B_{r_1}(z_k)\} \\
&\geq \sup\{x\cdot y - C|y - z_k|^{1+\beta} : y \in B_{r_1}(z_k)\} \\
&= \sup\{x\cdot (y - z_k) - C|y - z_k|^{1+\beta} : y \in B_{r_1}(z_k)\} + x\cdot z_k \\
  &\geq  C_1|x|^{1+a'}+l_kx_n
\end{align*}
for  \( a'=\frac{1}{\beta} > 1 \),
where \( C_1 \) is a constant independent of \( k \),  the first inequality follows from \eqref{alpha2}, and the last inequality
follows from a direct computation. It follows from the above estimate that
\begin{equation}\label{vust3}
u_{k1}(x) - l_k x_n \geq C_1 |x|^{1 + a'}\quad\forall\, x\in B_{r_2}(0) \cap \overline{\Omega_{k1}}.
\end{equation} 

By \eqref{zklev1}, we have that 
\[ B_{\frac{1}{C}}(z_k) \cap \overline{\Omega^*_{k1}} \subset S_1[v_k] \cap \overline{\Omega^*_{k1}}. \]
Then, the desired inclusion \eqref{vust} follows from \eqref{vust3} and \eqref{vust8}.

The proof of \eqref{vust2} is analogous. Indeed, Lemma \ref{vest1} also holds for \( v_{k} \) near \( \hat{z}_k := -z_k \) if we normalise \( S_{h_k}[v](z_k) \cap \Omega^*_2 \) instead of \( S_{h_k}[v](\hat z_k) \cap \Omega^*_1 \). It is important to note that both \eqref{vust} and \eqref{vust2} are affine invariant.
\end{proof}

By \eqref{beta1} and the duality between \( u_{k1} \) and \( v_k \) (see \eqref{dual222}), we have
\begin{align*}
u_{k1}(x) &= \sup\{x\cdot y - v_k(y) : y \in B_{r_1}(z_k) \cap \Omega_{k1}^*\} \\
&\leq \sup\{x\cdot y -  C|y - z_k|^{1+a} : y \in B_{r_1}(z_k) \cap \Omega_{k1}^*\} \\
&\leq \sup\{x\cdot y - C|y - z_k|^{1+a} : y \in \mathbb{R}^n\} \\
&= \sup\{x\cdot (y - z_k) - C|y - z_k|^{1+a} : y \in  \mathbb{R}^n\} + x\cdot z_k \\
&= C_1|x|^{1+\beta'}+l_kx_n
\end{align*}
for any \( x \in B_{r_2}(0)\cap\Omega_{k1} \), where \( \beta' = \frac{1}{a} \), and the last equality follows from the fact
that the Legendre transform of \( |x|^{1+a} \) is \( c|x|^{1+\frac{1}{a}} \) for some constant \( c>0 \).
Hence, it follows that \( u_{k1}(x) - l_kx_n \leq C_1|x|^{1+\beta'} \) for any \( x \in B_{r_2}(0)\cap\overline{\Omega_{k1}} \).

By the above inequality, for sufficiently large \( M \), we have
\begin{equation}\label{vust4}
  u_{k1}(x) - l_kx_n \leq \frac{1}{4N}\quad\text{ whenever }\ x \in B_{\frac{1}{M}}(0) \cap \overline{\Omega_{k1}}
\end{equation}
for sufficiently large $k.$
 By \eqref{vust8} and choosing \( M \) sufficiently large, we may also ensure that
\begin{equation}\label{cvk11}
  B_{\frac{1}{M}}(0) \cap \overline{\Omega_{k1}} \subset Dv_k(S_1[v_k] \cap \overline{\Omega^*_{k1}})
\end{equation}
for \( k \) sufficiently large.

Recall that \(\nu_k(x)\) is the unit inner normal of \(\Omega_{k1}\) at \(x \in \mathcal{F}_k \subset \partial \Omega_{k1}\). 
Let \(G_k\) be a connected component of
\(\{x \in B_1(0) \cap \mathcal{F}: \nu_k(x) \cdot e_n > 0\}\) such that \(0 \in G_k\), and 
\[ G_k  = \{ x = (x', x_n) \in \mathbb{R}^n : x_n = \rho_k(x'),\ \ x' \in G'_k \} \] 
for some 
\( C^{1,\beta} \) function \( \rho_k \) satisfying \( \rho_k(0) = 0 \) and \( D\rho_k(0) = 0 \), where \( G'_k \subset \mathbb{R}^{n-1} \) is the orthogonal projection of \( G_k \) onto the plane \( \{x_n = 0\} \). Note that \( G'_k \) is a connected open set in $\mathbb{R}^{n-1}$ containing the origin.

If \( p \in \mathcal{F}_k\cap B_{\frac{1}{M}}(0) \cap Dv_k(\overline{S_1[v_k]\cap\Omega^*_{k2}}) \), 
we have that \( Du_{k2}(p)\in \overline{S_1[v_k] \cap \Omega^*_{k2}} \).
By \eqref{cvk11}, we have that \( Du_{k1}(p)\in S_1[v_k] \cap \overline{\Omega^*_{k1}} \).
Since the width of \( S_1[v_k] \cap \Omega^*_{ki} \) (for \( i = 1, 2 \)) in the \( e_1 \) direction is bounded by a universal constant \( C \), it follows that \[ \left|\left(Du_{k1}(p) - Du_{k2}(p)\right)\cdot e_1\right|\leq C.\]
By Lemma \ref{limitdist} we have that \( |Du_{k1}(p) - Du_{k2}(p)|\geq l_k \) provided \( k \) is sufficiently large.
Hence, it follows that
\begin{equation}\label{vust15}
  \left| \nu_k(p) \cdot e_1 \right| = \left| \frac{Du_{k1}(p) - Du_{k2}(p)}{|Du_{k1}(p) - Du_{k2}(p)|} \cdot e_1 \right| \leq \frac{C}{l_k}
\end{equation}
for some  constant \( C \) independent of \( k \), provided \( k \) is sufficiently large.
By Lemma~\ref{limitdist}, we have that \(\text{diam}\left(S_1[v_k] \cap \Omega^*_{k1}\right) \ll l_k\) and \(\text{diam}\left(S_1[v_k] \cap \Omega^*_{k2}\right) \ll l_k\) for sufficiently large \(k\), which implies that 
\begin{equation}\label{nukp1}
\nu_k(p) \cdot e_n > \frac{1}{2} \quad \forall\, p \in \mathcal{F}_k \cap B_{\frac{1}{M}}(0) \cap Dv_k(\overline{S_1[v_k] \cap \Omega^*_{k2}}).
\end{equation}
From inequality~\eqref{vust15}, if we further assume that \( p \in G_k \), we can deduce that
\begin{equation}\label{7yue}
  \left| D_1\rho_k(p) \right| \leq \frac{C}{l_k}
\end{equation}
for sufficiently large \( k \).

Let
\[ p_t = (-t, 0, \ldots, 0, \rho_k(-t, 0, \ldots, 0)) \]
denote a point in \( G_k \cap B_1(0) \cap\text{span}\{e_1, e_n\}\).
Now, define
\begin{equation}
s_{k0} := \sup \left\{ s \mid -te_1\in G'_k,\ p_t \in G_k\cap B_{\frac{1}{2M}}(0) \cap Dv_k\left(\overline{S_1[v_k] \cap \Omega^*_{k2}}\right) \ \forall\,t\in[0,s] \right\}.
\end{equation}
We aim to establish a lower bound for $s_{k0}$ in the following lemma.

\begin{lemma}\label{ests}
For sufficiently large \( k \), we have
\begin{equation}\label{dudu}
s_{k0} \geq \min\left\{\frac{1}{4M}, \frac{1}{4NC}\right\},
\end{equation}
where \( C \) is a positive constant independent of \( k \).
\end{lemma}

\begin{proof}
Denote \( s_{k0} \) by \( s_0 \) for simplicity. Up to a subsequence, we may assume that
\( p_t \) converges to a point \( p_{s_0} \in B_{\frac{1}{M}}(0) \cap \mathcal{F}_k \) as \( t\rightarrow s_0 \).
For any \( t \in (0, s_0) \), we have that 
\[
p_{t} \in \mathcal{F}_k \cap B_{\frac{1}{M}}(0) \cap Dv_k(\overline{S_1[v_k] \cap \Omega^*_{k2}}).
\]
Then, it follows from \eqref{nukp1} that \( \nu_k(p_t) \cdot e_n \geq \frac{1}{2} \) for \( t \in (0, s_0) \) and \( k \) sufficiently large.
By continuity, we have
\[
\nu_k(p_{s_0}) \cdot e_n \geq \frac{1}{2},
\]
which implies that 
\begin{equation}\label{pscon1}
p_{s_0} \in G_k.
\end{equation}

Suppose that \eqref{dudu} does not hold; then, \( s_0 < \frac{1}{4M} \).
By estimate \eqref{7yue}, we have 
\[
|\rho_k(-t, 0, \ldots, 0)| \leq \frac{Ct}{l_k} \leq \frac{C}{4Ml_k} \leq \frac{1}{4M} \quad\text{ for } t < s_0
\] 
provided \( k \) is sufficiently large.
Now, for \( t \in (0, s_0) \), we have that 
\[
|p_t| = \sqrt{t^2 + |\rho_k(-t, 0, \ldots, 0)|^2} \leq \frac{1}{2\sqrt{2}M}.
\]
By continuity, this implies that 
\begin{equation}\label{pscon2}
p_{s_0} \in B_{\frac{1}{2M}}(0)
\end{equation}
provided \( k \) is sufficiently large.

We \emph{claim} that 
\begin{equation}\label{sknot2}
p_{s_0} \notin Dv_k\left(S_1[v_k] \cap \overline{\Omega^*_{k2}}\right)
\end{equation}
for \( k \) sufficiently large.
Otherwise, by \eqref{pscon1} and \eqref{pscon2}, for sufficiently large \( k \), we could extend \( t \) beyond \( s_0 \), namely,
\[
p_{s_0+\epsilon} := (-s_0-\epsilon, 0, \ldots, 0, \rho_k(-s_0-\epsilon, 0, \ldots, 0)) \in G_k \cap B_{\frac{1}{2M}}(0) \cap Dv_k\left(\overline{S_1[v_k] \cap \Omega^*_{k2}}\right),
\]
which would be contradicting the definition of \( s_0 \).

Combining \eqref{sknot2} and \eqref{vust2},
we obtain
\begin{equation}\label{dest}
u_{k2}(p_{s_0}) + l_kp_{s_0} \cdot e_n \geq \frac{1}{N}.
\end{equation}
On the other hand, since $s_0<\frac{1}{4NC}$ by assumption, by \eqref{vust4} and \eqref{7yue} we obtain
	\begin{align}
		u_{k1}(p_{s_0}) &\leq \frac{1}{4N} + l_kp_{s_0} \cdot e_n, \nonumber \\
			&\leq \frac{1}{4N} + l_k\frac{C}{l_k}s_0, \label{ests11} \\
			&\leq \frac{1}{2N}. \nonumber 
	\end{align}
Because \( u_{k1} = u_{k2} \) on \( \mathcal{F}_k \), by \eqref{ests11} we have
\begin{align*}
u_{k2}(p_{s_0}) + l_kp_{s_0} \cdot e_n &= u_{k1}(p_{s_0}) + l_kp_{s_0} \cdot e_n \\
&\leq \frac{1}{2N} + \frac{1}{4N} < \frac{1}{N},
\end{align*}
which contradicts \eqref{dest}.
\end{proof}

Let \( t_0 := \min\left\{\frac{1}{4M}, \frac{1}{4NC}\right\} \). By \eqref{7yue} and Lemma \ref{ests}, we deduce that
\[
|\rho_k(-t_0e_1)| = |p_{t_0} \cdot e_n| \leq \frac{C}{l_k}t_0 \rightarrow 0 \quad \text{as} \ k \rightarrow \infty.
\]
Define
\[
\tilde{u}_k(x) := \sup\left\{ x \cdot y - v_k(y + z_k) : y + z_k \in S_1[v_k] \cap \Omega^*_{k1} \right\}.
\]
Since $\left(S_1[v_k] \cap \Omega^*_{k1}\right)-\{x_k\}\subset B_C(0)$ for some constant $C$ independent of $k,$ we have that $\|\tilde u_k\|_{_{Lip}}\leq C.$ 
Note that \( \tilde{u}_k = u_{k1} - l_kx_n \) on \( B_{\frac{1}{2M}}(0) \cap \overline{\Omega_{k1}} \). With the assumption that \( \Omega^*_{k1} \subset \{x_1 \geq 0\} \), we have \( D_1\tilde{u}_k \geq 0 \). It implies that \( \tilde{u}_k(-t_0e_1)\leq  \tilde{u}_k(0)=0\).
 Consequently,
\begin{align*}
0\leq u_{k1}(p_{t_0}) - l_kp_{t_0} \cdot e_n &= \tilde{u}_k(p_{t_0}) \\
&\leq \tilde{u}_k(-t_0e_1) + \|\tilde u_k\|_{_{Lip}}p_{t_0} \cdot e_n \\
&\leq Cp_{t_0} \cdot e_n \\
&\leq \frac{C}{l_k}t_0 \rightarrow 0 \quad \text{as} \ k \rightarrow \infty,
\end{align*}
which contradicts \eqref{vust3} for sufficiently large \( k \).

\subsection{Case  \textbf{III}: $\nu\cdot \nu^*=0,\ \nu \cdot \hat\nu^* > 0$.}

After an appropriate affine transformation, we can assume that \(\nu\) is parallel to \(\hat{\nu}^*\). 
Indeed, we only need to choose an affine transform \( A \) such that \( A\nu \) is parallel to \( (A^t)^{-1}\hat{\nu}^* \),
see \cite[(3.2)]{CLW3} for the existence of such a transformation. Note that after the transformation, 
\(\frac{(A^t)^{-1}\nu}{|(A^t)^{-1}\nu|}\) as the inner unit normal of \( A\Omega_1 \) at \( 0 \) is still orthogonal to
\(\frac{A\nu^*}{|A\nu^*|}\) as the inner unit normal of \( (A^t)^{-1}\Omega_1^* \) at \( Ay_0 \).

As in \eqref{calf1} and \eqref{ostar1}, we have that \(0 \in \mathcal{F}\), \(y_0 = Du_1(0) = le_n \in \partial\Omega^*_1\), and \(\hat{y}_0 = Du_2(0) = -le_n \in \partial\Omega^*_2\), where \(l\) is a positive universal constant. 
Additionally, we have the unit normals \(\nu = \nu(0) = e_n\), \(\nu^* = \nu^*(y_0) = e_1\), and \(\hat{\nu}^* = \hat{\nu}^*(\hat{y}_0) = -e_n\). By subtracting a constant, we may assume \(u_i(0) = 0\) for \(i = 1, 2\) and \(v(y_0) = v(\hat{y}_0) = 0\).

In this case, Lemma \ref{tc141} is applicable, leading to
\begin{equation*}
u_2(x)-\hat y_0\cdot e_n \leq C_\epsilon |x|^{2-\epsilon} \quad\text{for all } x \in B_{r_0}(0).
\end{equation*}
Given that \(v = u_2^*\) within \(\Omega_2^*\), it follows that
\begin{equation*}
v(y) \geq C_\epsilon |y - \hat y_0|^{2+\epsilon} \quad \forall\, y \in \Omega^*_2.
\end{equation*}
It is noteworthy that in Case \textbf{I}, the specific conditions \(\hat\nu \cdot \hat\nu^* = 0\) and the non-coplanarity of \(\nu, \nu^*, \hat\nu^*\) are exclusively used in the proofs of Lemmas \ref{decayrate2} and \ref{saddle}. 

The proof of Lemma \ref{decayrate2} relies on the inequality
\begin{equation*}
v(y) \geq C_\epsilon |y - \hat y_0|^{3+\epsilon}, \quad \forall y \in \Omega_2^*.
\end{equation*}
In the present scenario, this inequality remains valid since \(|y - \hat y_0|^{2+\epsilon}\) dominates \(|y - \hat y_0|^{3+\epsilon}\) when \(y\) is close to \(\hat y_0\).

By inclusions \eqref{tanc2} and \eqref{equi0}, we can find a point \(p \in S_h[v]\)  such that \(p_1 \gtrsim h^{\frac{1}{2} + \epsilon}\), creating a contradiction with \eqref{saddle1}. Hence, Lemma \ref{saddle} is still valid in this particular scenario. As a result, the remainder of the proof from Case \textbf{I} can be replicated to eliminate the possibility of the current case.

\section{$C^{2,\alpha}$ estimates}
\label {S5}
Let \( \Omega, \Omega_1, \Omega_2, \Omega^*, \Omega^*_1, \Omega^*_2, \mathcal{F}, f, g, u, u_1, u_2, v, \nu, \hat\nu, \nu^*, \) and \( \hat\nu^* \) be as defined in Section \ref{S3}. Fix a point \( x_0 \in \mathcal{F} \) and denote \( y_0 = Du_1(x_0) \). 
Thanks to Proposition \ref{oblique2}, \( \nu \cdot \nu^* > 0 \), we can make a change of coordinates such that \( x_0 = y_0 = 0 \), and for some sufficiently small \( r > 0 \), 
\[
\Omega_1 \cap B_r(0) = \{x_n > \rho(x')\} \cap B_r(0) \quad \text{and} \quad \Omega_1^* \cap B_r(0) = \{x_n > \rho^*(x')\} \cap B_r(0),
\]
where \( \rho \) is a \( C^{1,\beta} \) function satisfying \( \rho(0) = 0 \) and \( D\rho(0) = 0 \), and \( \rho^* \) is a \( C^2 \) uniformly convex function with \( \rho^*(0) = 0 \) and \( D\rho^*(0) = 0 \). Here, \( x' = (x_1, \ldots, x_{n-1}) \).

Given the above configuration (which follows from the obliqueness estimate), and using the tangential $C^{1,1-\epsilon}$ estimate for $v$ (Corollary \ref{co21}), the duality relation \eqref{dual222}, and the arguments in the proofs of Lemmas 3.3 and 3.4 in \cite{CLW3}, we can conclude that:
\begin{lemma} \label{tc12}
For any $\epsilon>0$ small, there exists a constant $C_\epsilon$ such that:
\begin{align}
    u_1(x) &\geq C_\epsilon |x'|^{2+\epsilon} \quad \forall\, x \in \Omega_1\cap B_r(0), \label{ustconv} \\
    u_1(te_n) &\leq C_\epsilon |t|^{2-\epsilon} \quad \forall\, t\ \text{satisfying that \(|t|\) is sufficiently small}. \nonumber
\end{align}
\end{lemma}

In \cite{CLW3}, by utilising $\rho(x') \leq C|x'|^2$ (derived from the interior ball property for optimal partial transport), a uniform density estimate for $u_1$ can then be derived. 
In the current situation, we only know that $\mathcal{F}$ is $C^{1,\beta}$ regular, namely $\rho(x') \leq C|x'|^{1+\beta}$ for some $\beta\in(0,1)$.
Below, we modify the proof of \cite[Lemma 3.5]{CLW3} to obtain the uniform density estimate under this weaker condition. 
\begin{lemma} \label{ud3}
For any $h>0$ small, one has 
\[
\frac{|S^c_h[u_1] \cap \Omega_1|}{|S^c_h[u_1]|} \geq \delta_0,
\]
where $\delta_0 > 0$ is a constant independent of $h$, and $S^c_h[u_1] = S^c_h[u_1](0)$ denotes the centred section of $u_1$ with height $h$.
\end{lemma}

\begin{proof}
Consider the intersections of $\partial S_h^c[u_1]$ with the $x_n$-axis, denoted by $z = se_n$ and $\tilde{z} = -\tilde{s}e_n$, where $s, \tilde{s} > 0$. Since $S_h^c[u_1]$ is centred at $0$, one has $s\approx \tilde{s}$. Consequently, we deduce that either $u_1(z) \geq Ch$ or $u_1(\tilde{z}) \geq Ch$ holds. 
Then, by Lemma \ref{tc12}, we have
\begin{equation}\label{gs1}
s \approx \tilde{s} \geq C_\epsilon h^{\frac{1}{2} + \epsilon}
\end{equation}
for any arbitrarily small $\epsilon > 0$.

Utilising \eqref{equi0} and Lemma \ref{tc12}, we derive that
\begin{equation}\label{gs2}
S^c_h[u_1] \cap \Omega_1 \subset S_{Ch}[u_1] \cap \Omega_1 \subset \left\{x : |x'| < C_\epsilon h^{\frac{1}{2} - \epsilon} \right\}.
\end{equation}
Recall the bound $|\rho(x')| \leq C|x'|^{1+\beta}$. For any point $x\in S_h^c[u_1] \cap \{x_n \geq C'h^{\frac{1}{2} + \frac{1}{4}\beta}\} \cap \overline{\Omega_1}$, inequality \eqref{gs2} ensures that $|x'| < C_\epsilon h^{\frac{1}{2} - \epsilon}$, which, in turn, implies 
	$$\rho(x') < C' h^{(\frac{1}{2} - \epsilon)(1 + \beta)} \leq C'h^{\frac{1}{2} + \frac{1}{4}\beta} \leq x_n$$ 
provided $\epsilon$ is sufficiently small, where $C' = 2CC_\epsilon^{1 + \beta}$. 
Consequently, $x\in \Omega_1$. This leads us to conclude that
\begin{equation}\label{cptinc121}
S_h^c[u_1] \cap \{x_n \geq C'h^{\frac{1}{2} + \frac{1}{4}\beta}\} \cap \overline{\Omega_1} \Subset \Omega_1.
\end{equation}

Assume, for the sake of contradiction, that there exists a point 
	$$x \in S_h^c[u_1] \cap \{x_n \geq C'h^{\frac{1}{2} + \frac{1}{4}\beta}\} \setminus \Omega_1.$$
Connecting $x$ and $z$ with a line segment, we observe that it intersects $\partial \Omega_1$ at some point $y$. As both $x$ and $z$ belong to $S_h^c[u_1] \cap \{x_n \geq C'h^{\frac{1}{2} + \frac{1}{4}\beta}\}$, the convexity of $u_1$ implies $y \in S_h^c[u_1] \cap \{x_n \geq C'h^{\frac{1}{2} + \frac{1}{4}\beta}\} \cap \partial \Omega_1$, contradicting \eqref{cptinc121}. Therefore, we deduce that
\begin{equation}\label{gs3}
S_h^c[u_1] \cap \{x_n \geq C'h^{\frac{1}{2} + \frac{1}{4}\beta}\} \subset \Omega_1,
\end{equation}
which implies that a ``substantial" portion of $S^c_h[u_1]$ is contained within $\Omega_1$.

Invoking John's Lemma, we identify an ellipsoid $E$ centred at the origin such that $E \subset S_h^c[u_1] \subset CE$. Due to \eqref{gs1}, we have $s \gg C'h^{\frac{1}{2} + \frac{1}{4}\beta}$ for sufficiently small $h$. Considering the convexity of $S_h^c[u_1]$ and \eqref{gs3}, we estimate
\begin{align*}
\left| S_h^c[u_1] \cap \Omega_1 \right| &\geq \left| S_h^c[u_1] \cap \{x_n \geq C'h^{\frac{1}{2} + \frac{1}{4}\beta}\} \right| \\
&\geq \left| E \cap  \{x_n \geq C'h^{\frac{1}{2} + \frac{1}{4}\beta}\} \right| \\
&\geq c\frac{s - h^{\frac{1}{2} + \frac{1}{4}\beta}}{s} |E| \\
&\geq \frac{1}{2}c |S_h^c[u_1]|,
\end{align*}
where $c > 0$ depends only on $n$. Consequently, we obtain
\[
\frac{|S^c_h[u_1] \cap \Omega_1|}{|S^c_h[u_1]|} \geq \frac{c}{2}.
\]
\end{proof}

Once having the uniform density estimate, by the proof and techniques employed in \cite[Lemma 3.8 and Corollary 3.9]{CLW3}, we can obtain the subsequent estimates for $u_1$.

\begin{lemma}\label{tc14}
For any $\epsilon > 0$ small, there exists a constant $C_\epsilon > 0$ such that:
\begin{enumerate}
    \item Tangential $C^{1, 1-\epsilon}$ estimate: $B_{C_\epsilon h^{\frac{1}{2}+\epsilon}}(0) \cap \{x_n = 0\} \subseteq S_h^c[u_1].$
    \item Almost $C^{1, 1}$ estimate:  $u_1(x) \leq C_\epsilon |x|^{2 - \epsilon}$ $\forall\, x \in B_{r_0}(0)$.
  \item Strict convexity in $\Omega_1$:  $u_1(x) \geq C_\epsilon |x|^{2 + \epsilon}$ $\forall\, x \in \Omega_1 \cap B_{r_0}(0)$.  
    \item Gradient bound: $|Du_1(x)| \leq C_\epsilon |x|^{1 - \epsilon}$ $\forall\, x \in B_{\frac{r_0}{2}}(0).$
\end{enumerate}
Here, $r_0 > 0$ is a suitably chosen small constant.
\end{lemma}

Thanks to Proposition \ref{oblique2}, $\hat{\nu} \cdot \hat{\nu}^* > 0$, similarly we can establish analogous estimates for the function $u_2$. 
As before, by a change of coordinates we may assume \( x_0 = \hat y_0 = 0 \) and \[
\Omega_2 \cap B_r(0) = \{x_n > \tilde\rho(x')\} \cap B_r(0) \quad \text{and} \quad \Omega_2^* \cap B_r(0) = \{x_n > \tilde\rho^*(x')\} \cap B_r(0),
\]
where \(\tilde \rho \) is a \( C^{1,\beta} \) function satisfying \( \tilde\rho(0) = 0 \) and \( D\tilde\rho(0) = 0 \), and \( \tilde\rho^* \) is a \( C^2 \) uniformly convex function with \( \tilde\rho^*(0) = 0 \) and \( D\tilde\rho^*(0) = 0 \). 
\begin{lemma}\label{tc141}
For any $\epsilon > 0$ small, there exists a constant $C_\epsilon > 0$ such that:
\begin{enumerate}
    \item Tangential $C^{1, 1-\epsilon}$ estimate: $B_{C_\epsilon h^{\frac{1}{2}+\epsilon}}(0) \cap \{x_n = 0\} \subseteq S_h^c[u_2].$
    \item Almost $C^{1, 1}$ estimate:  $u_2(x) \leq C_\epsilon |x|^{2 - \epsilon}$ $\forall\, x \in B_{r_0}(0)$.
  \item Strict convexity in $\Omega_2$:  $u_2(x) \geq C_\epsilon |x|^{2 + \epsilon}$ $\forall\, x \in \Omega_2 \cap B_{r_0}(0)$.  
    \item Gradient bound: $|Du_2(x)| \leq C_\epsilon |x|^{1 - \epsilon}$ $\forall\, x \in B_{\frac{r_0}{2}}(0).$
\end{enumerate}
Here, $r_0 > 0$ denotes a sufficiently small constant.
\end{lemma}

\begin{proof}[Proof of Theorem \ref{main1}]
Lemmas \ref{tc14} and \ref{tc141} imply that the functions \( u_i \) for \( i = 1, 2 \) are  \( C^{1,1-\epsilon} \) along the singular set \( \mathcal{F} \) for any sufficiently small \( \epsilon > 0 \). Subsequently, from \eqref{gn1}, it follows that \( \mathcal{F} \) itself possesses \( C^{1,1-\epsilon} \) regularity.

Having established the $C^{1,1-\epsilon}$ regularity of both $u_i$ and $\mathcal{F}$, we are now in a position to apply the perturbation argument tailored for the optimal partial transport problem, as delineated in \cite[Section 4]{CLW3}. This argument ensures that the regularity of $u_i$ can be further improved to $C^{2,\alpha}$ along $\mathcal{F}$. From \eqref{gn1} once more, we conclude that the singular set $\mathcal{F}$ is also of class $C^{2,\alpha}$.
\end{proof}

\bibliographystyle{amsplain}

\end{document}